\documentclass[journal]{IEEEtran}
\usepackage{graphics}
\usepackage{graphicx}
\usepackage{color}
\usepackage{mathrsfs}
\usepackage{amssymb}
\usepackage{amsmath}
\usepackage{boxedminipage}
\usepackage{stmaryrd}
\usepackage{multirow}
\usepackage{booktabs}
\usepackage{accents}

\ifCLASSINFOpdf
\else
\fi
\begin{document}
\newcounter{comp1}

\newtheorem{definition}{Definition}
\newtheorem{proposition}{Proposition}
\newtheorem{example}{Example}
\newtheorem{method}{Method}
\newtheorem{lemma}{Lemma}
\newtheorem{theorem}{Theorem}
\newtheorem{corollary}{Corollary}
\newtheorem{assumption}{Assumption}
\newtheorem{algorithm}{Algorithm}
\newtheorem{remark}{Remark}
\newcommand{\fig}[1]{\begin{figure}[hbt]
                  \vspace{1cm}
                  \begin{center}
                  \begin{picture}(15,10)(0,0)
                  \put(0,0){\line(1,0){15}}
                  \put(0,0){\line(0,1){10}}
                  \put(15,0){\line(0,1){10}}
                  \put(0,10){\line(1,0){15}}
                  \end{picture}
                  \end{center}
                  \vspace{.3cm}
                  \caption{#1}
                  \vspace{.5cm}
                  \end{figure}}
\newcommand{\Axk}{A(x^k)}
\newcommand{\Aumb}{\sum_{i=1}^{N}A_{i}u_{i}-b}
\newcommand{\Kk}{K^k}
\newcommand{\Kki}{K_{i}^{k}}
\newcommand{\Aukmb}{\sum_{i=1}^{N}A_{i}u_{i}^{k}-b}
\newcommand{\Au}{\sum_{i=1}^{N}A_{i}u_{i}}
\newcommand{\Aukpmb}{\sum_{i=1}^{N}A_{i}u_{i}^{k+1}-b}
\newcommand{\nab}{\nabla^2 f(x^k)}
\newcommand{\xk}{x^k}
\newcommand{\ubk}{\overline{u}^k}
\newcommand{\uhk}{\hat u^k}
\def\QEDclosed{\mbox{\rule[0pt]{1.3ex}{1.3ex}}} 
\def\QEDopen{{\setlength{\fboxsep}{0pt}\setlength{\fboxrule}{0.2pt}\fbox{\rule[0pt]{0pt}{1.3ex}\rule[0pt]{1.3ex}{0pt}}}}
\def\QED{\QEDopen} 
\def\proof{\par\noindent{\em Proof.}}
\def\endproof{\hfill $\Box$ \vskip 0.4cm}
\newcommand{\RR}{\mathbf R}
%
\title{Stochastic Primal-Dual Coordinate Method with Large Step Size for Composite Optimization with Composite Cone-constraints}
%
%
%

\author{Daoli~Zhu~
        and~Lei~Zhao
\thanks{Manuscript received February 16, 2019; revised.}
\thanks{Daoli Zhu was with Antai College of Economics and Management and Sino-US Global Logistics
Institute, Shanghai Jiao Tong University, 200030 Shanghai, China (e-mail: dlzhu@sjtu.edu.cn)}
\thanks{Lei Zhao was with the Antai College of Economics and Management, Shanghai
Jiao Tong University, 200030 Shanghai, China (e-mail: l.zhao@sjtu.edu.cn)}
}

%
%

\markboth{Journal of \LaTeX\ Class Files,~Vol.~14, No.~8, August~2015}%
{Shell \MakeLowercase{\textit{et al.}}: Bare Demo of IEEEtran.cls for IEEE Journals}
%



\maketitle

\begin{abstract}
We introduce a stochastic coordinate extension of the first-order primal-dual method studied by Cohen and Zhu (1984) and Zhao and Zhu (2018) to solve Composite Optimization with Composite Cone-constraints (COCC). In this method, we randomly choose a block of variables based on the uniform distribution. The linearization and Bregman-like function (core function) to that randomly selected block allow us to get simple parallel primal-dual decomposition for COCC. We obtain almost surely convergence and $O(1/t)$ expected convergence rate in this work. The high probability complexity bound is also derived in this paper.
\end{abstract}


%
\IEEEpeerreviewmaketitle
\section{Introduction}\label{Intro}
Motivated by recent applications in big data analysis, there has been an explosive growth in interest in the design and analysis of block coordinate descent type (BCD-type) methods for large-scale convex optimization. (see~\cite{LiOsher, QinScheinberg, WuLange, YunToh}) In these applications, the datasets used for computation are very big and are often distributed in different locations. It is often impractical to assume that optimization algorithms can traverse an entire dataset once in each iteration, because doing so is either time consuming or unreliable, and often results in low resource utilization due to necessary synchronization among different computing units (e.g., CPUs, GPUs, and cores) in a distributed computing environment.  On the other hand, BCD-type algorithms can make progress by using information obtained from a randomly selected subset of data and, thus, provide much flexibility for their implementation in the aforementioned distributed environments. The main advantage of BCD-type method is to reduce the complexity and memory requirements per iteration. These benefits are increasingly important for very-large scale problem.\\
\indent In this paper, we consider the nonlinear convex cone-constrained optimization problem known as a Composite Optimization with Composite Cone-constrains (COCC):
\begin{equation}\label{Prob:general-function}
\begin{array}{lll}
\mbox{(P):}  &\min       & G(u)+J(u)      \\
             &\rm {s.t}  & \Theta(u)=\Omega(u)+\Phi(u)\in -\mathbf{C} \\
             &           & u\in \mathbf{U}
\end{array}
\end{equation}
where $G$ is a convex smooth function on the closed convex set $\mathbf{U}\subset \RR^{n}$ and $J$ is a convex, possibly nonsmooth function on $\mathbf{U}\subset \RR^{n}$. $\Omega$ is a smooth and $\Phi$ is a possibly nonsmooth mapping from $\RR^{n}$ to $\RR^{m}$. $\Omega(u)$ and $\Phi(u)$ are $\mathbf{C}$-convex and $\mathbf{C}$ is a nonempty closed convex cone in $\RR^{m}$ with vertex at the origin, that is, $\alpha\mathbf{C}+\beta\mathbf{C}\subset \mathbf{C}$, for $\alpha,\beta\geq 0$. It is obvious that when $\mathring{\mathbf{C}}$ (the interior of $\mathbf{C}$) is nonempty, the constraint $\Theta(u)\in -\mathbf{C}$ corresponds to an inequality constraint. The case $\mathbf{C}=\{0\}$ corresponds to an equality constraint. $\mathbf{C}^{*}$ denotes the conjugate cone i.e. $\mathbf{C}^*=\{y|\langle y,x\rangle\geq 0, \forall x\in\mathbf{C}\}$. We note that COCC has full composite structure.\\ Assume that both $J(u)=\sum\limits_{i=1}\limits^{N}J_i(u_i)$ and $\Phi(u)=\sum\limits_{i=1}\limits^{N}\Phi_{i}(u_{i})$ are additive respect to following space decomposition:
\begin{equation}\label{spacedecomposition}
\mathbf{U}=\mathbf{U}_1\times\mathbf{U}_2\cdots\times\mathbf{U}_N, u_i\in\mathbf{U}_i\subset \RR^{n_i}~\mbox{and}~ \sum\limits_{i=1}\limits^{N}n_i=n.
\end{equation}
\subsection{Related works}
\indent For problems without constraints, there are two variations of BCD discussed the most by researchers. The first variation is on block-choosing strategy. One common approach for block choosing is cyclic strategy. Tseng~\cite{Tseng} proved the convergence of a BCD of cyclic strategy. Luo and Tseng~\cite{LuoZeng} and Wang and Lin~\cite{WangLin} proved local and global linear convergence under specific assumptions respectively. The other approach is randomized strategy. Nesterov~\cite{Nesterov} studied the convergence rate of randomized BCD for convex smooth optimization. Richt\'arik and Tak\'a\v c~\cite{RT} and Lu and Xiao~\cite{LuXiao} extended Nesterov's technique to composite optimization. The point read to evaluate the gradient in each iteration is the second variation of BCD. If the read points have different "ages", this type of BCD called asynchronous BCD; otherwise, it is called synchronous BCD. All the variants of BCD reviewed above are synchronous BCD. Liu and Wright~\cite{LiuWright1} and Liu et. al.~\cite{LiuWright2} established the convergence rate of asynchronous BCD for composite optimization and convex smooth optimization without constraints, respectively.\\
\indent For problems with constraints, there are only a few works. Gao et. al.~\cite{GaoXuZhang} proposed a coordinate-type method for problems with linear coupling constraints. Necoara and Patrascu~\cite{Necoara} proposed a random coordinate descent algorithm for an optimization problem with one linear constraint. Xu and Zhang~\cite{Xu18} analyzed primal-dual coordinate type method for a linear constrained strongly convex problem. Moreover, Xu~\cite{Xu19} proposed an asynchronous primal-dual coordinate-type method for linear constrained problems. For problem with nonlinear constraints, Xu~\cite{Xu17} proposed a coordinate-type method for problem with nonlinear inequality constraints. To the best of our knowledge, there is no primal-dual coordinate convergence rate results for COCC.
\subsection{Main contributions and outline of this paper}
\indent In this paper, we propose a Stochastic Primal-Dual Coordinate with Large step size (SPDCL) method based on the variant auxiliary problem principle (Zhao and Zhu~\cite{ZhaoZhu2017}) for COCC. In this method, we randomly update one block of variables based on the uniform distribution. The sequence generated by our algorithm is proved to converge to an optimal solution of problem (P) with probability $1$. The expected $O(1/t)$ convergence rate is also obtained for problem (P) under the convexity assumptions. The probability complexity bound is also derived in this paper.\\
\indent The rest of this paper is organized as follows. Section~\ref{Pre} is devoted to technical preliminaries. The updating scheme of SPDCL for (P) is presented in Section~\ref{SPDC}. In Section~\ref{convergence}, we establish the convergence. In Section~\ref{rate}, expected $O(1/t)$ sub-linear convergence rate and the high probability complexity bound are established.
\section{Preliminaries}\label{Pre}
In this section, we first provide some preliminaries that are useful for our further discussions and then summarize some notations and assumptions to be used. We denote $\langle \cdot \rangle $ and $\| \cdot \|$ as the inner product and Euclidean norm of vector, respectively.
\subsection{Notations and assumptions}
Throughout this paper, we make the following standard assumptions for Problem (P):
\begin{assumption}\label{assump1}
{\rm
\noindent \begin{itemize}
\item[(i)] $J$ is a convex, l.s.c function such that $\mathbf{dom}J\cap \mathbf{U}\neq\emptyset$, $J$ is not necessary differentiable. $J$ is subgradientiable and has linear bounded subgradients in $\mathbf{U}$, that is
    \begin{eqnarray*}
    \exists c_1>0, c_2>0,\;\forall u\in\mathbf{U},\;\forall r\in\partial J(u),\;\|r\|\leq c_1\|u\|+c_2.
    \end{eqnarray*}
\item[(ii)] $G$ is a convex and differentiable with its derivative Lipschitz of constant $B_{G}$.
\item[(iii)] $\Omega$ is $\mathbf{C}$-convex mapping from $\mathbf{U}$ to $\mathbf{C}$, where $\forall u,v\in \mathbf{U}$, $\forall\alpha\in [0,1]$,
\begin{equation}\label{Theta_C_convex}
\Omega(\alpha u+(1-\alpha)v)-\alpha\Omega(u)-(1-\alpha)\Omega(v)\in - \mathbf{C}.
\end{equation}
Moreover, the derivative of $\Omega$ exists and meets the following condition: $\exists T\in\mathbf{C}$ such that $\forall u,v\in \mathbf{U}$,
\begin{equation}\label{UB-1}
\langle\nabla\Omega(u)-\nabla\Omega(v),u-v\rangle-\|u-v\|^2T\in-\mathbf{C}.
\end{equation}
\item[(iv)] $\Psi$ is $\mathbf{C}$-convex mapping from $\mathbf{U}$ to $\mathbf{C}$.\\
\item[(v)] $\Theta(u)$ is Lipschitz with constant $\tau$ on an open subset $\mathcal{O}$ containing $\mathbf{U}$, where
\begin{equation}\label{Theta_Lipschitz}
\forall u,v\in \mathcal{O}, \|\Theta(u)-\Theta(v)\|\leq\tau\|u-v\|.
\end{equation}
\item[(vi)] Constraint Qualification Condition. When $\mathring{\mathbf{C}}\neq\emptyset$, we assume that
\begin{equation}\label{Theta_CQC_neq}
\mbox{{\bf CQC:}}\qquad\Theta(\mathbf{U})\cap(-\mathring{\mathbf{C}})\neq\emptyset.\qquad
\end{equation}
For the case $\mathbf{C}=\{0\}$, we assume that
$0\in\mbox{interior of}~\Theta(\mathbf{U})$.
\item[(vii)] There exists at least one saddle point for Lagrangian of {\rm(P)}.
\end{itemize}
}
\end{assumption}
\indent Condition (i)-(iv) guarantee that (P) is a convex problem. The CQC condition (vi) implies that the Lagrangian dual function is coercive and the dual optimal solution set is bounded~\cite{CohenZ}. Furthermore, the following subsection gives augmented Lagrangian and first-order primal-dual decomposition algorithm for (P).
\subsection{Augmented Lagrangian and first-order primal-dual decomposition algorithm}
In this subsection, the Lagrangian of (P) is defined as:
\begin{equation}\label{func:L}
L(u,p)=(G+J)(u)+\langle p,\Theta(u)\rangle,
\end{equation}
and a saddle point $(u^*,p^*)\in \mathbf{U}\times \mathbf{C}^*$ is such that
\begin{equation}
\forall u\in \mathbf{U},\; \forall p\in \mathbf{C}^{*}:\; L(u^{*},p)\leq L(u^{*},p^{*})\leq L(u,p^{*}). \label{saddle point:L}
\end{equation}
\indent Under Assumption~\ref{assump1}, there exist saddle points of $L$ on $\mathbf{U}\times \mathbf{C}^*$. The dual function $\psi$ is defined as
\begin{equation*}
\psi(p)=\begin{cases}\displaystyle\min_{u\in \mathbf{U}}L(u,p)  &\forall p\in\mathbf{C}^*\\
                     \displaystyle-\infty  &\mbox{otherwise.}\end{cases}
\end{equation*}
\indent The function $\psi$ is concave and sub-differentiable. Using dual function $\psi(p)$, we consider the primal-dual pair of nonlinear convex cone optimization:
\begin{equation*}
\begin{array}{lllllll}
  \mbox{(P):} &\min      & (G+J)(u)                 &\qquad&\mbox{(D):}    &\max      & \psi(p) \\
              &\rm {s.t} & \Theta(u)\in -\mathbf{C} &\qquad&               &\rm {s.t} & p\in\mathbf{C}^*\\
              &          & u\in\mathbf{U}           &\qquad&               &          &
\end{array}
\end{equation*}
\indent The following theorem characterizes a saddle point optimality condition for the primal and dual problem.
\begin{theorem}\label{theo:Lag}
A solution $(u^*,p^*)$ with $u^*\in \mathbf{U}$ and $p^*\in \mathbf{C}^*$ is a saddle point for the Lagrangian function $L(u,p)$ if and only if
\begin{itemize}
\item[{\rm(i)}] $L(u^*,p^*)=\min\limits_{u\in \mathbf{U}}L(u,p^*)$\\
           or the following variational inequality holds: $\forall u\in\mathbf{U}$, $$\langle\nabla G(u^*),u-u^*\rangle+J(u)-J(u^*)+\langle p^*,\Theta(u)-\Theta(u^*)\rangle\geq 0;$$
\item[{\rm(ii)}] $\Theta(u^*)\in -\mathbf{C}$;
\item[{\rm(iii)}] $\langle p^*,\Theta(u^*)\rangle=0$.
\end{itemize}
Moreover, $(u^*,p^*)$ is a saddle point if and only if $u^*$ and $p^*$ are, respectively, optimal solutions to the primal and dual problems (P) and (D) with no duality gap, that is, with $(G+J)(u^*)=\psi(p^*)$.
\end{theorem}
\indent Now we take a trick by introducing slack variables which help problem (P) come back to problem with equality constraints. Namely, the problem (P) is converted into the equivalent problem with equality constraints as follows
\begin{equation*}
\begin{array}{lllllll}
  \mbox{(P$_1$):} &\min\limits_{\xi\in-\mathbf{C}}      & (G+J)(u) \\
              &\rm {s.t} & \Theta(u)-\xi=0 \\
              &          & u\in\mathbf{U}  \\
\end{array}
\end{equation*}
\indent The augmented Lagrangian for this problem is
\begin{equation}\label{L_xi}
\overline{L}_\gamma(u,\xi,p)=(G+J)(u)+\langle p,\Theta(u)-\xi\rangle+\frac{\gamma}{2}\|\Theta(u)-\xi\|^2
\end{equation}
\indent The augmented Lagrangian associated with problem (P) is defined as
\begin{equation}\label{L_gamma}
L_\gamma(u,p)\triangleq\min_{\xi\in-\mathbf{C}}\overline{L}_\gamma(u,\xi,p)=(G+J)(u)+\varphi(\Theta(u),p),
\end{equation}
where $\varphi(\Theta(u),p)=[\|\Pi\big{(}p+\gamma\Theta(u)\big{)}\|^2-\|p\|^2]/2\gamma$ and $\Pi$ is a projection on to $\mathbf{C}^*$.\\
\indent The augmented Lagrangian dual function is as following:
\begin{eqnarray}
\forall p\in \RR^m, \psi_\gamma(p)&=&\min_{u\in\mathbf{U}}L_\gamma(u,p)\nonumber\\
                                  &=&\min_{u\in\mathbf{U}}(G+J)(u)+\varphi(\Theta(u),p).
\end{eqnarray}
\indent Using $\psi_\gamma(p)$, we obtain new primal-dual pair of nonlinear convex cone optimization
\begin{equation*}
\begin{array}{lllllll}
  \mbox{(P):} &\min      & (G+J)(u)                 &\qquad&\mbox{(D$_\gamma$):}    &\max      & \psi_\gamma(p) \\
              &\rm {s.t} & \Theta(u)\in -\mathbf{C} &\qquad&                      &\rm {s.t} & p\in\mathbf{R}^m\\
              &          & u\in\mathbf{U}           &\qquad&                      &          &
\end{array}
\end{equation*}
\indent The following theorem shows that function $\varphi(\theta,p)$, dual function $\psi_\gamma(p)$ and augmented Lagrangian $L_\gamma(u,p)$ have some useful properties.
\begin{theorem}\label{theo:Lag_2}
Suppose Assumption~\ref{assump1} holds for problem (P). Then we have
\begin{itemize}
\item[{\rm(i)}] The function $\varphi(\theta,p)$ is convex in $\theta$ and concave in $p$.
\item[{\rm(ii)}] $\varphi$ is differentiable in $\theta$ and $p$ and one has
\begin{eqnarray*}
&&\nabla_{\theta}\varphi(\theta, p)=\Pi(p+\gamma\theta),\\
&&\nabla_{p}\varphi(\theta, p)=[\Pi(p+\gamma\theta)-p]/\gamma,\\
&&\varphi(\theta, p)=[\|\Pi(p+\gamma\theta)\|^2-\|p\|^2]/2\gamma.
\end{eqnarray*}
\item[{\rm(iii)}] $\psi_\gamma(p)$ is concave and differentiable in $p$, and $\nabla\psi_\gamma(p)=[\Pi(p+\gamma\Theta(\hat{u}(p)))-p]/\gamma$, where $\hat{u}(p)\in\hat{\mathbf{U}}(p)=\{u\in\mathbf{U}|u=\arg\min\limits_{u\in\mathbf{U}}L_\gamma(u,p)\}$.
\item[{\rm(iv)}] $L$ and $L_\gamma$ have the same sets of saddle points $\mathbf{U}^*\times\mathbf{P}^*$ respectively on $\mathbf{U}\times\mathbf{C}^*$ and $\mathbf{U}\times\mathbf{R}^m$.
\item[{\rm(v)}] $L_\gamma$ is stable in $u$, that is $\forall p^*\in \mathbf{P}^*, \hat{\mathbf{U}}(p^*)=\mathbf{U}^*$.
\end{itemize}
\end{theorem}
\indent Moreover, next lemma will give another property of augmented Lagrangian term.
\begin{lemma}\label{lemma:bound0} For all $p\in\mathbf{C}^*$, $p'\in\RR^m$ and $u\in\RR^n$, we have that
\begin{eqnarray*}
L(u,p)-L_{\gamma}(u,p')\leq\frac{1}{2\gamma}\|p-p'\|^2,
\end{eqnarray*}
or
\begin{eqnarray*}
\langle p,\Theta(u)\rangle-\varphi\big{(}\Theta(u),p'\big{)}\leq\frac{1}{2\gamma}\|p-p'\|^2.
\end{eqnarray*}
\end{lemma}
\begin{proof}
\begin{eqnarray*}
&&L_{\gamma}(u,p')-L(u,p)+\frac{1}{2\gamma}\|p-p'\|^2\nonumber\\
&=&\varphi\big{(}\Theta(u),p'\big{)}-\langle p,\Theta(u)\rangle+\frac{1}{2\gamma}\|p-p'\|^2\nonumber\\
&=&\frac{\|\Pi\big{(}p'+\gamma\Theta(u)\big{)}\|^2-\|p'\|^2}{2\gamma}-\langle p,\Theta(u)\rangle\nonumber\\
&&+\frac{1}{2\gamma}\big{[}\|p'\|^2-2\langle p,p'\rangle+\|p\|^2\big{]}\nonumber\\
&&\qquad\quad\;\mbox{(by expression of $\varphi\left(\Theta(u),p\right)$ in Theorem~\ref{theo:Lag_2})}\nonumber\\
&=&\frac{1}{2\gamma}\big{[}\|\Pi\big{(}p'+\gamma\Theta(u)\big{)}\|^2-2\langle p,p'+\gamma\Theta(u)\rangle+\|p\|^2\big{]}\nonumber\\
&=&\frac{1}{2\gamma}\big{[}\|\Pi\big{(}p'+\gamma\Theta(u)\big{)}\|^2-2\langle p,\Pi\big{(}p'+\gamma\Theta(u)\big{)}\nonumber\\
&&+\Pi_{-\mathbf{C}}\big{(}p'+\gamma\Theta(u)\big{)}\rangle+\|p\|^2\big{]}\qquad\qquad\qquad\mbox{(by~\eqref{eq:Projecproperty5})}\nonumber\\
&=&\frac{1}{2\gamma}\|\Pi\big{(}p'+\gamma\Theta(u)\big{)}-p\|^2-\frac{1}{\gamma}\langle p,\Pi_{-\mathbf{C}}\big{(}p'+\gamma\Theta(u)\big{)}\rangle\nonumber\\
&\geq&0.\qquad\qquad\mbox{(by $\langle p,\Pi_{-\mathbf{C}}\big{(}p'+\gamma\Theta(u)\big{)}\rangle\leq0$, $\forall p\in\mathbf{C}^*$)}
\end{eqnarray*}
\end{proof}
\indent For the general COCC, the augmented Lagrangian method is an approach which can overcome the instability and nondifferentiability of the dual function of the Lagrangian. Furthermore, the augmented Lagrangian of a constrained convex program has the same solution set as the original constrained convex program. The augmented Lagrangian approach for equality-constrained optimization problems was introduced in Hestenes~\cite{Hestenes1969} and Powell~\cite{Powell1969}, and then extended to inequality-constrained problems by Buys~\cite{Buys1972}.\\
\indent Although the augmented Lagrangian approach (Uzawa algorithm) has several advantages, it does not preserve separability, even when the initial problem is separable. One way to decompose the augmented Lagrangian is ADMM (Fortin and Glowinski~\cite{ADMM1983}). ADMM can only handle convex problems with linear constraints and is not easily parallelizable. Another way to overcome this difficulty is the Auxiliary Problem Principle of augmented Lagrangian methods (APP-AL) (Cohen and Zhu~\cite{CohenZ}), which is a fairly general first-order primal-dual decomposition method based on linearization of the augmented Lagrangian in nonlinear convex cone programming with separable or nonseparable, smooth or nonsmooth constraints. Zhao and Zhu (2018)~\cite{ZhaoZhu2017} extend Cohen and Zhu (1984)~\cite{CohenZ}'s work to propose first-order primal-dual augmented Lagrangian methods for COCC as an algorithm (VAPP).\\
\noindent\rule[0.25\baselineskip]{0.5\textwidth}{1.5pt}
{\bf Variant Auxiliary Problem Principle for solving COCC (VAPP)}\\
\noindent\rule[0.25\baselineskip]{0.5\textwidth}{0.5pt}
{Initialize} $u^0 \in \mathbf{U}$ and $p^0\in \mathbf{C^*}$  \\
 \textbf{for} $k = 0,1,\cdots $, \textbf{do}
\begin{eqnarray*}
&&u^{k+1}\leftarrow\min_{u\in \mathbf{U}}\langle\nabla G(u^{k}), u \rangle + J(u)+ \langle\Pi_M(p^k+\gamma\Theta(u^k)),\nonumber\\
&&\qquad\qquad\qquad\qquad\qquad\nabla\Omega(u^k)u+\Phi(u)\rangle+\frac{1}{\epsilon^k}D(u,u^k);\label{primal_APP}\\
&&p^{k+1}\leftarrow\Pi_M\big{(}p^k+\gamma\Theta(u^{k+1})\big{)}.\label{dual_APP}
\end{eqnarray*}
\textbf{end for}\\
\noindent\rule[0.25\baselineskip]{0.5\textwidth}{1.5pt}
\indent where $D(u,v)=K(u)-K(v)-\langle\nabla K(v),u-v\rangle$ is a Bregman like function with $K$ is strongly convex and gradient Lipschitz. Zhao and Zhu (2018) shows the sequence $\{(u^k,p^k)\}$ generated by VAPP convergence to $(u^*,p^*)$ saddle point of $L$ over $\mathbf{U}\times\mathbf{C}^*$. Moreover, an $O(1/t)$ convergence rate is also proposed. In the era of big data, there has been a surge of interest in redesign of VAPP suitable for solving the huge optimization with available computing performance.
\subsection{The properties of projection on convex cone}
In this subsection, we introduce some properties of projection on convex sets (resp. convex cone) as preparations. These properties are used in the following sections.\\
\indent Let $\mathcal{S}$ be a nonempty closed convex set of $\RR^m$. For $x\in\RR^m$, we propose the projection $\Pi_{\mathcal{S}}(x)$ as a projection on $\mathcal{S}$. Then $\Pi_{\mathcal{S}}(x)$ is characterized by the following two conditions~\cite{Projectiononconvexsets}:
\begin{eqnarray}
&(i)&\qquad\langle y-\Pi_{\mathcal{S}}(x), x-\Pi_{\mathcal{S}}(x)\rangle\leq0, \forall y\in\mathcal{S};\label{eq:Projecproperty3}\\
&(ii)&\qquad\|\Pi_{\mathcal{S}}(x)-\Pi_{\mathcal{S}}(y)\|\leq\|x-y\|.\label{eq:Projecproperty4}
\end{eqnarray}
\indent Furthermore, the following proposition gives another property of projection operator which is used for convergence and convergence rate analysis.
\begin{proposition}\label{proposition}
\noindent For any $(x,y,z)\in\RR^{m\times m\times m}$, the projection operator $\Pi_{\mathcal{S}}$ satisfies
\begin{eqnarray}\label{eq:Projecproperty7}
&&2\langle\Pi_{\mathcal{S}}(z+x)-\Pi_{\mathcal{S}}(z+y),x\rangle\\
&\leq&\|x-y\|^2+\|\Pi_{\mathcal{S}}(z+x)-z\|^2-\|\Pi_{\mathcal{S}}(z+y)-z\|^2.\nonumber
\end{eqnarray}
\end{proposition}
\begin{proof} See~\cite{ZhaoZhu2017}.
\end{proof}
\indent Next, we consider the properties for projection on convex cone. Let $\mathbf{C}$ be a nonempty closed convex cone in $\mathbf{R}^m$ with vertex at the origin. $\mathbf{C}^{*}$ denotes the conjugate cone. Let $\Pi$ denote the projection on $\mathbf{C}^*$ and $\Pi_{-\mathbf{C}}$ denote the projection on $-\mathbf{C}$. The projection is characterized by the following conditions. (see Wierzbicki~\cite{Projection}):
\begin{eqnarray}
&(iii)  &\qquad y=\Pi(y)+\Pi_{-\mathbf{C}}(y), y\in\mathbf{R}^m;\label{eq:Projecproperty5}\\
&(iv) &\qquad \langle\Pi(y), \Pi_{-\mathbf{C}}(y)\rangle=0, y\in\mathbf{R}^m.\label{eq:Projecproperty6}
\end{eqnarray}
\subsection{The properties of differentiable functions and mappings}
\begin{lemma}\label{lemma:Lipschitz3point}
\noindent Let the function $f$ be convex and differentiable on $\mathbf{U}$.\\
{\rm(i)} If $f$ is strongly convex with constant $\beta_f$, then
\begin{equation}\label{LB}
\forall u,v\in \mathbf{U}, f(u)-f(v)\geq\langle\nabla f(v),u-v\rangle+\frac{\beta_f}{2}\|u-v\|^2.
\end{equation}
{\rm(ii)} If the derivative of $f$ is Lipschitz with constant $B_f$, then
\begin{equation}\label{UB}
\forall u,v\in \mathbf{U}, f(u)-f(v)\leq \langle\nabla f(v),u-v\rangle+\frac{B_f}{2}\|u-v\|^2,
\end{equation}
{\rm(iii)} Let $\Omega$ be a $\mathbf{C}$-convex mapping from $\mathbf{U}$ to $\mathbf{C}$. Suppose its derivative exists and meets the following condition: $\exists T\in\mathbf{C}$ such that
\begin{equation}\label{UB-1}
\forall u,v\in \mathbf{U}, \langle\nabla\Omega(u)-\nabla\Omega(v),u-v\rangle-\|u-v\|^2T\in-\mathbf{C},
\end{equation}
then $\forall u,v\in \mathbf{U}, \forall p\in\mathbf{C}^*$ we have
\begin{equation}\label{UB-2}
\langle p,\Omega(u)-\Omega(v)\rangle\leq\langle p,\nabla\Omega(v)(u-v)\rangle+\frac{\|p\|\cdot\|u-v\|^2}{2}T.
\end{equation}
\end{lemma}
\begin{proof}
The statements (i) and (ii) are classical; the proof is omitted (see Zhu and Marcotte~\cite{Zhu96}). For proof of (iii), see Cohen~\cite{Cohen80}.
\end{proof}
\section{Stochastic primal-dual coordinate method}\label{SPDC}
In this section, we propose a stochastic primal-dual coordinate descent algorithm to solve (P). Firstly, we introduce the core function $K(\cdot)$ satisfying the following assumption:
\begin{assumption}\label{assump2}
\rm{
$K$ is strongly convex with parameter $\beta$ and differentiable with its gradient Lipschitz continuous with parameter $B$ on $\mathbf{U}$.
}
\end{assumption}
\indent Additionally, let $D(u,v)=K(u)-K(v)-\langle\nabla K(v),u-v\rangle$ is a Bregman like function (core function)~\cite{Mirror, CohenZ}. From Assumption~\ref{assump2} we have: $\frac{\beta}{2}\|u-v\|^2\leq D(u,v)\leq\frac{B}{2}\|u-v\|^2$.\\
\indent Moreover, we assume that the parameter $\rho$ satisfy:
\begin{equation}\label{para-choice-convergence}
\rho=\frac{\gamma}{2N-1}.
\end{equation}
Let $\mu_0$ be a bound of dual optimal solution of (P), denote $\mu=\mu_0+1$. Let $\mathfrak{B}_{\mu} = \{p|\|p\|\leq\mu\}$. The estimation of $\mu_0$ can be found in~\cite{ZhaoZhu2017}. By using the projection $\mathcal{P}_\mu(\cdot)$ onto $\mathfrak{B}_{\mu}$, we introduce Stochastic Primal-Dual Coordinate Method with Large step size (SPDCL) for solving (P):\\
\noindent\rule[0.25\baselineskip]{0.5\textwidth}{1.5pt}
{\bf Stochastic Primal-Dual Coordinate Method with Large step size (SPDCL)}\\
\noindent\rule[0.25\baselineskip]{0.5\textwidth}{0.5pt}
{Initialize} $u^0 \in \mathbf{U}$, $p^0\in \mathbf{R}^m$ and $\epsilon^{-1}>0$  \\
\textbf{for} $k = 0,1,\cdots $, \textbf{do}
\begin{eqnarray}
&&\mbox{Set $\epsilon^k=\min\{\epsilon^{k-1},\frac{\beta}{2(B_G+\|q^k\|\cdot T+\gamma\tau^2)}\}$}\nonumber\\
&&\mbox{Choose $i(k)$ from $\{1,2,\ldots,N\}$ with equal probability}\nonumber\\
&&u^{k+1}\leftarrow\min_{u\in \mathbf{U}}\langle\nabla_{i(k)} G(u^{k}), u_{i(k)} \rangle + J_{i(k)}(u_{i(k)})\nonumber\\
&&+\langle\Pi(p^k+\gamma\Theta(u^k)),\nabla_{i(k)}\Omega(u^k)u_{i(k)}+\Phi_{i(k)}(u_{i(k)})\rangle\nonumber\\
&&\qquad\qquad\qquad\qquad\qquad\qquad\qquad\quad+\frac{1}{\epsilon^k}D(u,u^k);\label{primal}\\
&&p^{k+1}\leftarrow\mathcal{P}_\mu\bigg{(}p^{k}+\frac{\rho}{\gamma}\big{(}\Pi(p^k+\gamma\Theta(u^{k+1}))-p^k\big{)}\bigg{)},\qquad\label{dual}
\end{eqnarray}
\textbf{end for}\\
\noindent\rule[0.25\baselineskip]{0.5\textwidth}{1.5pt}
For the sake of brevity, let us set that $q^k=\Pi\big{(}p^k+\gamma\Theta(u^k)\big{)}$, $q^{k+1/2}=\Pi\big{(}p^k+\gamma\Theta(u^{k+1})\big{)}$ and $F=G+J$. Then the primal problem of algorithm can be expressed as
\begin{eqnarray}
\mbox{(AP$^k$)}\qquad\min_{u\in \mathbf{U}}\langle\nabla_{i(k)} G(u^{k}), u_{i(k)}\rangle + J_{i(k)}(u_{i(k)})\nonumber\\
+\langle q^k,\nabla_{i(k)}\Omega(u^k)u_{i(k)}+\Phi_{i(k)}(u_{i(k)})\rangle\nonumber\\
+\frac{1}{\epsilon^k}\big{[}K(u)-\langle\nabla K(u^k),u\rangle\big{]}.\label{eq:APk}
\end{eqnarray}
If we choose an additive Bregman like function (or core function) respect to the space decomposition~\eqref{spacedecomposition} that is
$$K(u)=\sum_{i=1}^NK_i(u_i).$$
Then problem (AP$^k$) is just a small optimization problem for selected block $i(k)$. Specifically, taking $K(u)=\sum\limits_{i=1}\limits^{N}\frac{\|u_i\|^2}{2}$ for (AP$^k$), we perform only a block proximal gradient update for block $i(k)$, where we linearize the coupled function $G(u)$ and augmented Lagrangian term $\varphi(\Theta(u),p)$ and add the proximal term to it. In the following sections, we will establish the convergence and convergence rate and probability complexity bounds of SPDCL.
\section{Convergence analysis}\label{convergence}
In this section, we will establish results about convergence of SPDCL. Before proceeding, we first give the generalized equilibrium reformulation of saddle point formulation~\eqref{saddle point:L}:\\
Find $(u^*, p^*)\in\mathbf{U}\times\mathbf{C}^*$ such that
\begin{equation}\label{VIS}
\mbox{(EP):}\qquad L(u^*, p)-L(u, p^*)\leq 0, \forall u\in\mathbf{U}, p\in\mathbf{C}^*.
\end{equation}
Obviously, bifunction $L(u',p)-L(u,p')$ is convex in $u'$ and linear in $p'$ for given $u\in\mathbf{U}$, $p\in\mathbf{C}^*$.\\
In algorithm SPDCL, the indices $i(k)$, $k=0,1,2,\ldots$ are random variables. After $k$ iterations, SPDCL method generates a random output $(u^{k+1}, p^{k+1})$. We denote by $\mathcal{F}_k$ is a filtration generated by the random variable $i(0),i(1),\ldots,i(k)$, i.e.,
$$\mathcal{F}_{k}\overset{def}{=}\{i(0),i(1),\ldots,i(k)\}, \mathcal{F}_{k}\subset\mathcal{F}_{k+1}.$$
Additionaly, we define that $\mathcal{F}=(\mathcal{F}_{k})_{k\in\mathbb{N}}$,  $\mathbb{E}_{\mathcal{F}_{k+1}}=\mathbb{E}(\cdot|\mathcal{F}_{k})$ is the condition expectation w.r.t. $\mathcal{F}_{k}$ and the condition expectation in term of $i(k)$ given $i(0),i(1),\ldots,i(k-1)$ as $\mathbb{E}_{i(k)}$. \\
Knowing $\mathcal{F}_{k-1}=\{i(0),i(1),\ldots,i(k-1)\}$, we have:
\begin{eqnarray}
&&\mathbb{E}_{i(k)}\langle\nabla_{i(k)}G(u^{k}),(u^{k}-u)_{i(k)}\rangle\nonumber\\
&=&\frac{1}{N}\langle\nabla G(u^{k}),u^{k}-u\rangle\geq\frac{1}{N}\big{[}G(u^k)-G(u)\big{]};\label{expectationG}\\ &&\mathbb{E}_{i(k)}\big{[}J_{i(k)}(u_{i(k)}^{k})-J_{i(k)}(u_{i(k)})\big{]}\nonumber\\
&=&\frac{1}{N}\big{[}J(u^{k})-J(u)\big{]};\label{expectationJ}\\
&&\mathbb{E}_{i(k)}\langle q^k,\nabla_{i(k)}\Omega(u^k)(u^{k}-u)_{i(k)}\rangle\nonumber\\
&=&\frac{1}{N}\langle q^k,\nabla\Omega(u^k)(u^{k}-u)\rangle\nonumber\\
&\geq&\frac{1}{N}\langle q^k,\Omega(u^k)-\Omega(u)\rangle.\label{expectationP_1}\\
&&\mathbb{E}_{i(k)}\langle q^k,\Phi_{i(k)}(u_{i(k)}^{k})-\Phi_{i(k)}(u_{i(k)})\rangle\nonumber\\
&=&\frac{1}{N}\langle q^k,\Phi(u^{k})-\Phi(u)\rangle.\label{expectationP_2}
\end{eqnarray}
Given $(u^*,p^*)$, for any $u,u'\in\mathbf{U}$ and $p,p'\in\mathbf{C}^*$, we construct the following function:
\begin{eqnarray*}
&&\Lambda^k(u,p,u',p')\nonumber\\
&=&D(u,u')+\frac{\epsilon^k}{2N\rho}\|p-p'\|^2+\frac{(N-1)\epsilon^k}{N}\big{[}L_{\gamma}(u',p')\nonumber\\
&&-L(u^*,p^*)\big{]}\nonumber\\
&=&D(u,u')+\frac{\epsilon^k}{2N\rho}\|p-p'\|^2+\frac{(N-1)\epsilon^k}{N}\big{[}L_{\gamma}(u',p')\nonumber\\
&&-L(u',p^*)+L(u',p^*)-L(u^*,p^*)\big{]}\nonumber\\
&\geq&D(u,u')+\frac{\epsilon^k}{2N\rho}\|p-p'\|^2-\frac{(N-1)\epsilon^k}{2N\gamma }\|p^*-p'\|^2\\
&&\qquad\qquad\qquad\qquad\qquad\mbox{(since Lemma~\ref{lemma:bound0} and~\eqref{saddle point:L})}.
\end{eqnarray*}
Specifically, we can show the function value of $\Lambda^k$ at ($u^*,p^*$) provides an upper bound for $\|u'-u^*\|^2$.
\begin{eqnarray}\label{eq:Lambda_2}
&&\Lambda^k(u^*,p^*,u',p')\nonumber\\
&\geq&D(u^*,u')+\frac{\epsilon^k}{2N\rho}\|p^*-p'\|^2-\frac{(N-1)\epsilon^k}{2N\gamma }\|p^*-p'\|^2\nonumber\\
&\geq&D(u^*,u')+\frac{\epsilon^k}{2\gamma}\|p^*-p'\|^2\nonumber\\
&\geq&\frac{\beta}{2}\|u^*-u'\|^2.
\end{eqnarray}
Additionally, since the SPDCL scheme guarantee that $\epsilon^{k+1}\leq\epsilon^k$, we have that
\begin{eqnarray}\label{eq:Lambda}
&&\Lambda^k(u^*,p^*,u',p')\nonumber\\
&=&D(u^*,u')+\frac{\epsilon^k}{2N\rho}\|p^*-p'\|^2+\frac{(N-1)\epsilon^k}{N}\big{[}L_{\gamma}(u',p')\nonumber\\
&&-L(u',p^*)+L(u',p^*)-L(u^*,p^*)\big{]}\nonumber\\
&=&D(u^*,u')+\frac{\epsilon^k}{2\gamma}\|p^*-p'\|^2+\frac{(N-1)\epsilon^k}{N}\big{[}\frac{1}{2\gamma}\|p^*-p'\|^2\nonumber\\
&&+L_{\gamma}(u',p')-L(u',p^*)+L(u',p^*)-L(u^*,p^*)\big{]}\nonumber\\
&\geq&D(u^*,u')+\frac{\epsilon^{k+1}}{2\gamma }\|p^*-p'\|^2+\frac{(N-1)\epsilon^{k+1}}{N}\big{[}\frac{1}{2\gamma}\|p^*-p'\|^2\nonumber\\
&&+L_{\gamma}(u',p')-L(u',p^*)+L(u',p^*)-L(u^*,p^*)\big{]}\nonumber\\
&&\mbox{(the last term is nonegtive by Lemma~\ref{lemma:bound0} and inequality~\eqref{saddle point:L})}\nonumber\\
&=&D(u^*,u')+\frac{\epsilon^{k+1}}{2N\rho}\|p^*-p'\|^2\nonumber\\
&&+\frac{(N-1)\epsilon^{k+1}}{N}\big{[}L_{\gamma}(u',p')-L(u^*,p^*)\big{]}\nonumber\\
&=&\Lambda^{k+1}(u^*,p^*,u',p')
\end{eqnarray}
Before the convergence analysis, we need the following lemma.
\begin{lemma}\label{lemma:bound1} {\bf (Global estimation of bifunction values)}
Let Assumption~\ref{assump1} and~\ref{assump2} hold, $\{(u^k,p^k)\}$ is generated by SPDCL, the parameter $\rho$  satisfy~\eqref{para-choice-convergence}. For all $u\in\mathbf{U}$ and $p\in\mathbf{C}^*\cap\mathfrak{B}_{\mu}$, $(u,p)$ could possibly be random, it holds that\\
\begin{eqnarray*}
\mbox{{\rm(i)}}&&\frac{\epsilon^k}{N}\mathbb{E}_{i(k)}\big{[}L(u^{k+1},q^k)-L(u,q^k)\big{]}\\
               &\leq&\Lambda^k(u,p,u^k,p^k)-\mathbb{E}_{i(k)}\Lambda^k(u,p,u^{k+1},p^{k+1})\\
               &&-\frac{\epsilon^k}{2N\rho}\left[\|p-p^k\|^2-\mathbb{E}_{i(k)}\|p-p^{k+1}\|^2\right]\\
               &&+\mathbb{E}_{i(k)}\big{[}\frac{(N-1)\rho\epsilon^k}{N\gamma^2}\|q^{k+1/2}-p^k\|^2\\
               &&-\frac{\beta-\epsilon^k(B_G+\|q^k\|T+\frac{N-1}{N}\gamma\tau^2)}{2}\|u^k-u^{k+1}\|^2\big{]};\\
\mbox{{\rm(ii)}}&&\frac{\epsilon^k}{N}\mathbb{E}_{i(k)}\big{[}L(u^{k+1},p)-L(u^{k+1},q^{k})\big{]}\\
&\leq&\frac{\epsilon^k}{2N\rho}\big{[}\|p-p^k\|^2-\mathbb{E}_{i(k)}\|p-p^{k+1}\|^2\big{]}\\
&&+\mathbb{E}_{i(k)}\big{[}\frac{(1-N)\rho\epsilon^k}{N\gamma^{2}}\|q^{k+1/2}-p^k\|^2\nonumber\\
&&-\frac{\epsilon^k}{2N\gamma}\|q^k-p^k\|^2+\frac{\epsilon^k\frac{1}{N}\gamma\tau^2}{2}\|u^k-u^{k+1}\|^2\big{]};
\end{eqnarray*}
\begin{eqnarray*}
\mbox{{\rm(iii)}}&&\frac{\epsilon^k}{N}\mathbb{E}_{i(k)}\big{[}L(u^{k+1},p)-L(u,q^{k})\big{]}\\
&\leq&\Lambda^k(u,p,u^k,p^k)-\mathbb{E}_{i(k)}\Lambda^k(u,p,u^{k+1},p^{k+1})\\
&&-\mathbb{E}_{i(k)}\big{[}\frac{\beta}{4}\|u^k-u^{k+1}\|^2-\frac{\epsilon^k}{2N\gamma}\|q^k-p^k\|^2\big{]};\\
\mbox{{\rm(iv)}}&&\frac{1}{N}\big{[}L(u^{k},p)-L(u,q^{k})\big{]}\\
&\leq&\frac{B}{2\epsilon^k}\mathbb{E}_{i(k)}\|u^k-u^{k+1}\|^2\\
&&+\mathbb{E}_{i(k)}h_1(\epsilon^k,u,p,u^k,u^{k+1},q^k)\|u^k-u^{k+1}\|\\
&&+\mathbb{E}_{i(k)}h_2(p,p^k,p^{k+1})\|p^{k+1}-p^k\|,
\end{eqnarray*}
where\\
$h_1(\epsilon^k,u,p,u^k,u^{k+1},q^k)=\frac{B}{\epsilon^k}\|u-u^{k+1}\|+[\|\nabla G(u^k)\|+c_1\|u^k\|+c_2+\tau \|q^k\|]+\frac{\tau}{N}\|p-q^k\|$\\
and\\
$h_2(p,p^k,p^{k+1})=\frac{1}{2N\rho}\|2p-p^{k+1}-p^k\|$.
\end{lemma}
\begin{proof}
The proof of this lemma is left in Appendix.
\end{proof}
Based Lemma~\ref{lemma:bound1}, we establish the following convergence analysis of SPDCL.
\begin{theorem}[Almost surely convergence]\label{theo:convergence}
Let assumptions of Lemma~\ref{lemma:bound1} hold, then
\begin{itemize}
\item[{\rm(i)}] $\sum\limits_{k=0}\limits^{+\infty}\mathbb{E}_{i(k)}\frac{\beta}{4}\|u^k-u^{k+1}\|^2<+\infty$ $\mbox{a.s.}$ and $\sum\limits_{k=0}\limits^{+\infty}\frac{\epsilon^k}{2N\gamma}\|q^k-p^{k}\|^2<+\infty$ $\mbox{a.s.}$;
\item[{\rm(ii)}] The sequence $\{u^{k}\}$ generated by SPDCL is almost surely bounded;
\item[{\rm(iii)}] Every cluster point of $\{(u^{k},p^k)\}$ almost surely is a saddle point of Lagrangian of (P).
\end{itemize}
\end{theorem}
\begin{proof}
\begin{itemize}
\item[(i)] Take $u=u^*$ and $p=p^*$ in statement (iii) of Lemma~\ref{lemma:bound1}, we have
\begin{eqnarray}\label{eq:bound2_1_7}
&&\Lambda^k(u^*,p^*,u^k,p^k)\nonumber\\
&\geq&\mathbb{E}_{i(k)}\Lambda^{k}(u^{*},p^*,u^{k+1},p^{k+1})+S_k\nonumber\\
&\geq&\mathbb{E}_{i(k)}\Lambda^{k+1}(u^{*},p^*,u^{k+1},p^{k+1})+S_k.\\
&&\qquad\qquad\qquad\qquad\qquad\qquad\mbox{(by~\eqref{eq:Lambda})}\nonumber
\end{eqnarray}
By the definition of saddle point and assumption~\eqref{para-choice-convergence}, $(u^*,p^*)$ is solution of (EP),\\  $S_k=\mathbb{E}_{i(k)}\bigg{[}\frac{\epsilon^k}{N}\big{[}L(u^{k+1},p^*)-L(u^*,q^{k})\big{]}+\frac{\beta}{4}\|u^k-u^{k+1}\|^2+\frac{\epsilon^k}{2N\gamma}\|q^k-p^{k}\|^2\bigg{]}$\\ is positive. From~\eqref{eq:Lambda_2}, we have that $\Lambda^k(u^*,p^*,u^k,p^k)$ is nonnegative.\\
By the Robbins-Siegmund Lemma~\cite{RS}, we obtain that $\lim\limits_{k\rightarrow+\infty}\Lambda^k(u^*,p^*,u^k,p^k)$ almost surely exists, $\sum\limits_{k=0}\limits^{+\infty}\mathbb{E}_{i(k)}\frac{\beta}{4}\|u^k-u^{k+1}\|^2<+\infty$ $\mbox{a.s.}$ and $\sum\limits_{k=0}\limits^{+\infty}\frac{\epsilon^k}{2N\gamma}\|q^k-p^{k}\|^2<+\infty$ $\mbox{a.s.}$.\\
\item[(ii)] Since $\lim\limits_{k\rightarrow+\infty}\Lambda^k(u^{*},p^*,u^k,p^k)$ almost surely exists, thus $\Lambda^k(u^{*},p^*,u^k,p^k)$ is almost surely bounded.
Thanks~\eqref{eq:Lambda_2} it implies the sequence $\{u^k\}$ is almost surely bounded.
\item[(iii)] From statement (ii), we have that the sequence $\{u^k\}$ is almost surely bounded. Together with the SPDCL scheme guarantees that the sequence $\{p^k\}$ is bounded. Therefore, there exists a positive number $\underline{\epsilon}$ such that $\epsilon^k\geq\underline{\epsilon}$ with probability 1. Then from statement (i) we have that
\begin{eqnarray*}
\sum\limits_{k=0}\limits^{+\infty}\mathbb{E}_{i(k)}\frac{\beta}{4}\|u^k-u^{k+1}\|^2<+\infty\qquad\mbox{a.s.}
\end{eqnarray*}
and
\begin{eqnarray*}
\sum\limits_{k=0}\limits^{+\infty}\frac{\underline{\epsilon}}{2N\gamma}\|q^k-p^{k}\|^2\leq\sum\limits_{k=0}\limits^{+\infty}\frac{\epsilon^k}{2N\gamma}\|q^k-p^{k}\|^2<+\infty.\nonumber\\ \mbox{a.s.}
\end{eqnarray*}
It follows that
\begin{eqnarray}\label{eq:ukukp1}
\lim\limits_{k\to \infty}\mathbb{E}_{i(k)}\|u^{k}-u^{k+1}\|=0\nonumber\\
\mbox{and}\quad\lim\limits_{k\to \infty}\|q^{k}-p^{k}\|=0.
\end{eqnarray}
Since
\begin{eqnarray*}
\|p^k-p^{k+1}\|&\leq&\frac{\rho}{\gamma}\|q^{k+1/2}-p^k\|\nonumber\\
               &\leq&\frac{\rho}{\gamma}\left(\|q^{k+1/2}-q^k\|+\|q^{k}-p^k\|\right)\nonumber\\
               &\leq&\frac{\rho}{\gamma}\left(\gamma\tau\|u^k-u^{k+1}\|+\|q^{k}-p^k\|\right),
\end{eqnarray*}
then from~\eqref{eq:ukukp1}, we have almost surely
\begin{eqnarray}\label{eq:ukukp2}
\lim\limits_{k\to \infty}\mathbb{E}_{i(k)}\|p^{k}-p^{k+1}\|=0.
\end{eqnarray}
Let $\mathbb{W}_0$ denote the subset such that $\{u^k\}$ is not bounded, and let $\mathbb{W}_1$ denote the subset for which~\eqref{eq:ukukp1} does not hold: $\mathbb{P}(\mathbb{W}_0\cup\mathbb{W}_1)=0$. Pick some $\omega\notin\mathbb{W}_0\cup\mathbb{W}_1$. Since the sequence $\{u^k\}$ is almost surely bounded and $\{p^k\}$ is bounded, the sequence $\{(u^k,p^k)\}$ has cluster point. Considering a subsequence of $\{(u^k,p^k)\}$ almost surely converging toward $(\bar{u}(\omega),\bar{p}(\omega))$, let $\mathcal{N}(\bar{u})$ (resp. $\mathcal{N}(\bar{p})$) be neighbourhood of $\bar{u}(\omega)$ (resp. $\bar{p}(\omega)$). Together statement (iv) of Lemma~\ref{lemma:bound1}, the sequence $\{u^k\}$ is almost surely bounded, $\{p^k\}$ is bounded, almost surely $\underline{\epsilon}\leq\epsilon^k$ and $\epsilon^k\leq\epsilon^0$, we also have that there exists positive number $d_1$ and $d_2$ such that
\begin{eqnarray}\label{eq:limit}
&&\frac{1}{N}\big{[}L(u^{k},p)-L(u,q^{k})\big{]}\nonumber\\
&\leq&\frac{B}{2\epsilon^k}\mathbb{E}_{i(k)}\|u^k-u^{k+1}\|^2\nonumber\\
&&+\mathbb{E}_{i(k)}h_1(\epsilon^k,u,p,u^k,u^{k+1},q^k)\|u^k-u^{k+1}\|\nonumber\\
&&+\mathbb{E}_{i(k)}h_2(p,p^k,p^{k+1})\|p^{k+1}-p^k\|\nonumber\\
&\leq&\frac{B}{2\underline{\epsilon}}\mathbb{E}_{i(k)}\|u^k-u^{k+1}\|^2\nonumber\\
&&+\mathbb{E}_{i(k)}d_1\|u^k-u^{k+1}\|\nonumber\\
&&+\mathbb{E}_{i(k)}d_2\|p^{k+1}-p^k\|.\\
&&\forall (u,p)\in\mathcal{N}(\bar{u}(\omega))\times\mathcal{N}(\bar{p}(\omega))\subset\mathbf{U}\times(\mathbf{C}^*\cap\mathfrak{B}_{\mu})\nonumber\\
&&\qquad\qquad\qquad\qquad\qquad\qquad\quad\subset\mathbf{U}\times\mathbf{C}^*\nonumber
\end{eqnarray}
Passing to the limit of~\eqref{eq:limit}, it follows that $[L(\bar{u}(\omega),p)-L(u,\bar{p}(\omega))]\leq 0$, $\forall (u,p)\in\mathcal{N}(\bar{u}(\omega))\times\mathcal{N}(\bar{p}(\omega))\subset\mathbf{U}\times\mathbf{C}^*$. Therefore, $(\bar{u}(\omega),\bar{p}(\omega))$ is a saddle point of $L$ over $\mathcal{N}(\bar{u}(\omega))\times\mathcal{N}(\bar{p}(\omega))$. Since $L(u',p)-L(u,p')$ is convex in $(u',p')$, then $(\bar{u}(\omega),\bar{p}(\omega))$ is a saddle point of $L$ over $\mathbf{U}\times\mathbf{C}^*$.
\end{itemize}
\end{proof}
\section{Convergence rate analysis}\label{rate}
In this section we provide the convergence rate of SPDCL. For the sequence $\{(u^k,p^k)\}$ generated from Algorithm SPDCL, and any $t>0$ we define the average sequence
$$\bar{u}_{t}=\frac{\sum_{k=0}^{t}\epsilon^ku^{k+1}}{\sum_{k=0}^{t}\epsilon^k}\;\mbox{and}\;\bar{p}_{t}=\frac{\sum_{k=0}^{t}\epsilon^kq^k}{\sum_{k=0}^{t}\epsilon^k}.$$
\begin{theorem}\label{thm:primal_rate} {\bf(Expected primal suboptimality and expected feasibility)}\\
\indent Let Assumption~\ref{assump1} and~\ref{assump2} hold, $\{(u^k,p^k)\}$ is generated by SPDCL, the parameter $\rho$  satisfy condition~\eqref{para-choice-convergence}. Then we have that
\begin{itemize}
\item[{\rm(i)}]Boundness for expected vector:\\
    $\|\mathbb{E}_{\mathcal{F}_t}(\bar{u}_t)\|\leq\nu$\\
    where $\nu=\left(\frac{2\epsilon^0}{\beta\underline{\epsilon}}\Lambda(u^*,p^*,u^{0},p^0)\right)^{1/2}+\|u^*\|$;
\item[{\rm(ii)}]Global estimate of expect bifunction values:\\ $\mathbb{E}_{\mathcal{F}_t}\big{[}L(\bar{u}_t,p)-L(u,\bar{p}_t)\big{]}\leq\frac{Nh_3(u,p)}{\underline{\epsilon}(t+1)}$,\\
    where $h_3(u,p)=D(u,u^0)+\frac{N-1}{N}D(u^*,u^0)+\frac{\epsilon^0}{\gamma}\|p-p^0\|^2+\frac{(2N-1)(N-1)\epsilon^0}{N^2}\big{[}\frac{\|p^*-p^0\|^2}{2\gamma}+L_{\gamma}(u^0,p^0)-L(u^*,p^*)\big{]}$, $\forall u\in\mathbf{U}$, $p\in\mathbf{C}^*\cap\mathfrak{B}_{\mu}$, $(u,p)$ could possibly be random;
\item[{\rm(iii)}] Expected feasibility:\\ $\mathbb{E}_{\mathcal{F}_{t}}\|\Pi\left(\Theta(\bar{u}_t)\right)\|\leq\frac{Nd_3}{(\mu-\|p^*\|)\underline{\epsilon}(t+1)}$,\\
    where $d_3=\sup\limits_{\|p\|<\mu}h_3(u^*,p)$;
\item[{\rm(iv)}] Expected primal suboptimality:\\ $-\frac{\|p^*\|Nd_3}{(\mu-\|p^*\|)\underline{\epsilon}(t+1)}\leq\mathbb{E}_{\mathcal{F}_{t}}\left[F(\bar{u}_t)-F(u^*)\right]\leq\frac{Nd_3}{\underline{\epsilon}(t+1)}$.
\end{itemize}
\end{theorem}
\begin{proof}
\begin{itemize}
\item[{\rm(i)}] From statement (iii) of Lemma~\ref{lemma:bound1}, we obtain that
\begin{eqnarray*}\label{eq:rate_primal_1}
&&\frac{\epsilon^k}{N}\mathbb{E}_{i(k)}\big{[}L(u^{k+1},p)-L(u,q^{k})\big{]}\nonumber\\
&\leq&\Lambda^k(u,p,u^k,p^k)-\mathbb{E}_{i(k)}\Lambda^k(u,p,u^{k+1},p^{k+1})\nonumber\\
&&-\mathbb{E}_{i(k)}\bigg{[}\frac{\beta}{4}\|u^k-u^{k+1}\|^2+\frac{\epsilon^k}{2N\gamma}\|q^k-p^k\|^2\bigg{]}\nonumber\\
&\leq&\Lambda^k(u,p,u^k,p^k)-\mathbb{E}_{i(k)}\Lambda^k(u,p,u^{k+1},p^{k+1}).
\end{eqnarray*}
Taking expectation with respect to $\mathcal{F}_t$, $t>k$ for above inequality, we obtain that
\begin{eqnarray}\label{eq:rate_primal_2}
&&\mathbb{E}_{\mathcal{F}_t}\frac{\epsilon^k}{N}\big{[}L(u^{k+1},p)-L(u,q^{k})\big{]}\nonumber\\
&\leq&\mathbb{E}_{\mathcal{F}_t}\big{[}\Lambda^k(u,p,u^k,p^k)-\Lambda^k(u,p,u^{k+1},p^{k+1})\big{]}.\nonumber\\
\end{eqnarray}
Take $u=u^*\in\mathbf{U}$ and $p=p^*\in\mathbf{C}^*\cap\mathfrak{B}_{\mu}$ in~\eqref{eq:rate_primal_2} we have that
\begin{eqnarray}\label{eq:rate_primal_3}
&&\mathbb{E}_{\mathcal{F}_t}\Lambda^k(u^*,p^*,u^k,p^k)\nonumber\\
&\geq&\mathbb{E}_{\mathcal{F}_t}\Lambda^{k}(u^*,p^*,u^{k+1},p^{k+1})\nonumber\\
&\geq&\mathbb{E}_{\mathcal{F}_t}\Lambda^{k+1}(u^*,p^*,u^{k+1},p^{k+1})
\end{eqnarray}
Together with~\eqref{eq:Lambda_2} and~\eqref{eq:rate_primal_3}, we have
\begin{eqnarray}\label{eq:rate_primal_4_1}
\mathbb{E}_{\mathcal{F}_t}\frac{\beta}{2}\|u^{k+1}-u^*\|^2&\leq&\mathbb{E}_{\mathcal{F}_t}\Lambda^{k+1}(u^*,p^*,u^{k+1},p^{k+1})\nonumber\\
&\leq&\Lambda^0(u^*,p^*,u^{0},p^0).
\end{eqnarray}
From the convexity of $\|\cdot\|^2$ and $\epsilon^k$ is almost surely bounded below with $\underline{\epsilon}$ (by Theorem~\ref{theo:convergence}), we obtain that
\begin{eqnarray*}\label{eq:rate_primal_4_2}
\|\mathbb{E}_{\mathcal{F}_t}(\bar{u}_{t})-u^*\|^2&\leq&\mathbb{E}_{\mathcal{F}_t}\|\bar{u}_{t}-u^*\|^2\\
                                                 &\leq&\mathbb{E}_{\mathcal{F}_t}\frac{\sum_{k=0}^t\epsilon^k\|u^{k+1}-u^*\|^2}{\sum_{k=0}^t\epsilon^k}\\
                                                 &\leq&\frac{2\epsilon^0}{\beta\underline{\epsilon}}\Lambda^0(u^*,p^*,u^{0},p^0).
\end{eqnarray*}
Here comes the results.
\item[(ii)] Then from~\eqref{eq:rate_primal_2}, we obtain that
\begin{eqnarray*}\label{eq:rate_primal_5}
&&\mathbb{E}_{\mathcal{F}_t}\frac{\epsilon^k}{N}\big{[}L(u^{k+1},p)-L(u,q^{k})\big{]}\nonumber\\
&\leq&\mathbb{E}_{\mathcal{F}_t}\big{[}\Lambda^k(u,p,u^k,p^k)-\Lambda^k(u,p,u^{k+1},p^{k+1})\big{]}\nonumber\\
&=&\mathbb{E}_{\mathcal{F}_t}\big{[}\Lambda^k(u,p,u^k,p^k)-\Lambda^k(u,p,u^{k+1},p^{k+1})\big{]}\nonumber\\
&&+\frac{N-1}{N}\mathbb{E}_{\mathcal{F}_t}\bigg{\{}\big{[}\Lambda^k(u^*,p^*,u^k,p^k)-\Lambda^k(u^*,p^*,u^{k+1},p^{k+1})\big{]}\nonumber\\
&&+\big{[}\Lambda^k(u^*,p^*,u^{k+1},p^{k+1})-\Lambda^k(u^*,p^*,u^k,p^k)\big{]}\bigg{\}}
\end{eqnarray*}
From~\eqref{eq:rate_primal_3}, we have that $\mathbb{E}_{\mathcal{F}_t}[\Lambda^k(u^*,p^*,u^{k+1},p^{k+1})-\Lambda^k(u^*,p^*,u^k,p^k)]\leq0$, then by the definition of $\Lambda^k$, it follows
\begin{eqnarray*}
&&\mathbb{E}_{\mathcal{F}_t}\frac{\epsilon^k}{N}\big{[}L(u^{k+1},p)-L(u,q^{k})\big{]}\nonumber\\
&\leq&\mathbb{E}_{\mathcal{F}_t}\big{[}\Lambda^k(u,p,u^k,p^k)-\Lambda^k(u,p,u^{k+1},p^{k+1})\big{]}\nonumber\\
&&+\frac{N-1}{N}\mathbb{E}_{\mathcal{F}_t}\bigg{\{}\big{[}\Lambda^k(u^*,p^*,u^k,p^k)-\Lambda^k(u^*,p^*,u^{k+1},p^{k+1})\big{]}\nonumber\\
&=&\mathbb{E}_{\mathcal{F}_t}\bigg{\{}\big{[}D(u,u^k)-D(u,u^{k+1})\big{]}+\frac{\epsilon^k}{2N\rho}\big{[}\|p-p^k\|^2\nonumber\\
&&-\|p-p^{k+1}\|^2\big{]}+\frac{(N-1)\epsilon^k}{N}\big{[}L_{\gamma}(u^k,p^k)-L_{\gamma}(u^{k+1},p^{k+1})\big{]}\nonumber\\
&&+\frac{N-1}{N}\big{[}D(u^*,u^k)-D(u^*,u^{k+1})\big{]}\nonumber\\
&&+\frac{(2N-1)(N-1)\epsilon^k}{2N^2\gamma}\big{[}\|p^*-p^k\|^2-\|p^*-p^{k+1}\|^2\big{]}\nonumber\\
&&+\frac{(N-1)^2\epsilon^k}{N^2}\big{[}L_{\gamma}(u^k,p^k)-L_{\gamma}(u^{k+1},p^{k+1})\big{]}\bigg{\}}\nonumber\\
&=&\mathbb{E}_{\mathcal{F}_t}\bigg{\{}\big{[}D(u,u^k)-D(u,u^{k+1})\big{]}+\frac{\epsilon^k}{2N\rho}\big{[}\|p-p^k\|^2\nonumber\\
&&-\|p-p^{k+1}\|^2\big{]}+\frac{N-1}{N}\big{[}D(u^*,u^k)-D(u^*,u^{k+1})\big{]}\nonumber\\
&&+\frac{(2N-1)(N-1)\epsilon^k}{N^2}\bigg{[}\big{[}\frac{\|p^*-p^k\|^2}{2\gamma}+L_{\gamma}(u^k,p^k)\nonumber\\
&&-L(u^*,p^*)\big{]}-\big{[}\frac{\|p^*-p^{k+1}\|^2}{2\gamma}+L_{\gamma}(u^{k+1},p^{k+1})\nonumber\\
&&-L(u^*,p^*)\big{]}\bigg{]}\bigg{\}}
\end{eqnarray*}
By Lemma~\ref{lemma:bound0} we have that
\begin{eqnarray}\label{eq:positive}
&&\frac{\|p^*-p^k\|^2}{2\gamma}+L_{\gamma}(u^k,p^k)-L(u^*,p^*)\nonumber\\
&=&\frac{\|p^*-p^k\|^2}{2\gamma}+L_{\gamma}(u^k,p^k)-L(u^k,p^*)+L(u^k,p^*)\nonumber\\
&&-L(u^*,p^*)\nonumber\\
&\geq&0
\end{eqnarray}
Combine with $\epsilon^{k+1}\leq\epsilon^k$, we have that
\begin{eqnarray}\label{eq:rate_primal_7}
&&\mathbb{E}_{\mathcal{F}_t}\frac{\epsilon^k}{N}\big{[}L(u^{k+1},p)-L(u,q^{k})\big{]}\nonumber\\
&\leq&\mathbb{E}_{\mathcal{F}_t}\bigg{\{}\big{[}D(u,u^k)-D(u,u^{k+1})\big{]}\nonumber\\
&&+\frac{N-1}{N}\big{[}D(u^*,u^k)-D(u^*,u^{k+1})\big{]}\nonumber\\
&&+\frac{\epsilon^k}{2N\rho}\|p-p^k\|^2-\frac{\epsilon^{k+1}}{2N\rho}\|p-p^{k+1}\|^2\big{]}\nonumber\\
&&+\frac{(2N-1)(N-1)\epsilon^k}{N^2}\big{[}\frac{\|p^*-p^k\|^2}{2\gamma}+L_{\gamma}(u^k,p^k)\nonumber\\
&&-L(u^*,p^*)\big{]}-\frac{(2N-1)(N-1)\epsilon^{k+1}}{N^2}\big{[}\frac{\|p^*-p^{k+1}\|^2}{2\gamma}\nonumber\\
&&+L_{\gamma}(u^{k+1},p^{k+1})-L(u^*,p^*)\big{]}\bigg{\}}
\end{eqnarray}
Summing~\eqref{eq:rate_primal_7} over $k=1,2,...,t$, it follows that
\begin{eqnarray}\label{eq:rate_primal_5}
\mathbb{E}_{\mathcal{F}_t}\sum_{k=0}^{t}\frac{\epsilon^k}{N}\big{[}L(u^{k+1},p)-L(u,q^{k})\big{]}\leq h_3(u,p)
\end{eqnarray}
where $h_3(u,p)=D(u,u^0)+\frac{N-1}{N}D(u^*,u^0)+\frac{\epsilon^0}{\gamma}\|p-p^0\|^2+\frac{(2N-1)(N-1)\epsilon^0}{N^2}\big{[}\frac{\|p^*-p^0\|^2}{2\gamma}+L_{\gamma}(u^0,p^0)-L(u^*,p^*)\big{]}$.\\
Another hand, from the definition of $\bar{u}_t$ and $\bar{p}_t$, we have $\bar{u}_t\in\mathbf{U}$ and $\bar{p}_t\in\mathbf{C}^*\cap\mathfrak{B}_{\mu}$. From the convexity of set $\mathbf{U}$, $\mathbf{C}^*\cap\mathfrak{B}_{\mu}$ and the function $L(u',p)-L(u,p')$ is convex in $u'$ and linear in $p'$, for all $u\in\mathbf{U}$ and $p\in\mathbf{C}^*\cap\mathfrak{B}_{\mu}$, since $\epsilon^k$ is almost surely bounded below with $\underline{\epsilon}$ (by Theorem~\ref{theo:convergence}), we have that
\begin{eqnarray}\label{eq:rate_primal_6}
&&\mathbb{E}_{\mathcal{F}_t}\big{[}L(\bar{u}_{t},p)-L(u,\bar{p}_{t})\big{]}\nonumber\\
&\leq&\mathbb{E}_{\mathcal{F}_t}\frac{1}{\sum_{k=0}^t\epsilon^k}\sum_{k=0}^t\epsilon^k\big{[}L(u^{k+1},p)-L(u,q^k)\big{]}\nonumber\\
&\leq&\mathbb{E}_{\mathcal{F}_t}\frac{1}{\underline{\epsilon}(t+1)}\sum_{k=0}^t\epsilon^k\big{[}L(u^{k+1},p)-L(u,q^k)\big{]}\nonumber\\
&\leq&\frac{Nh_3(u,p)}{\underline{\epsilon}(t+1)}.
\end{eqnarray}
\item[{\rm(iii)}] If $\mathbb{E}_{\mathcal{F}_{t}}\|\Pi\left(\Theta(\bar{u}_t)\right)\|=0$, statement (ii) is obviously. Otherwise, $\mathbb{E}_{\mathcal{F}_{t}}\|\Pi\left(\Theta(\bar{u}_t)\right)\|\neq0$ i.e., there is set $\mathbb{W}_3$ such that $\mathbb{P}\{\omega\in\mathbb{W}_3|\|\Pi\left(\Theta(\bar{u}_t)\right)\|\neq0\}>0$. Let $\hat{p}$ be a random vector:
\begin{eqnarray}\label{eq:p_omega}
\hat{p}(\omega)=\left\{
\begin{array}{cc}
0 & \omega\notin\mathbb{W}_3 \\
\frac{\mu\Pi\big{(}\Theta(\bar{u}_t)\big{)}}{\|\Pi\big{(}\Theta(\bar{u}_t)\big{)}\|} & \omega\in\mathbb{W}_3.
\end{array}
\right.
\end{eqnarray}
Noted that for $\omega\notin\mathbb{W}_3$, we have $\hat{p}(\omega)=0$ and $\|\Pi\left(\Theta(\bar{u}_t)\right)\|=0$. Thus
\begin{eqnarray}\label{star0}
\langle\hat{p}(\omega),\Theta(\bar{u}_t)\rangle=\mu\|\Pi\left(\Theta(\bar{u}_t)\right)\|=0.
\end{eqnarray}
Otherwise, for $\omega\in\mathbb{W}_3$, we have that
\begin{eqnarray}\label{star1}
&&\langle \hat{p}(\omega),\Theta(\bar{u}_t)\rangle\nonumber\\
&=&\langle\frac{\mu\Pi\big{(}\Theta(\bar{u}_t)\big{)}}{\|\Pi\big{(}\Theta(\bar{u}_t)\big{)}\|},\Theta(\bar{u}_t)\rangle\nonumber\\
&=&\langle\frac{\mu\Pi\big{(}\Theta(\bar{u}_{t})\big{)}}{\|\Pi\big{(}\Theta(\bar{u}_{t})\big{)}\|},\Pi\big{(}\Theta(\bar{u}_{t})\big{)}+\Pi_{-\mathbf{C}}\big{(}\Theta(\bar{u}_{t})\big{)}\rangle~\mbox{(by~\eqref{eq:Projecproperty5})}\nonumber\\
&=&\mu\|\Pi\big{(}\Theta(\bar{u}_t)\big{)}\|.\qquad\qquad\mbox{(since~\eqref{eq:Projecproperty6})}
\end{eqnarray}
Together~\eqref{star0} and~\eqref{star1}, we have
\begin{eqnarray}\label{star}
\langle\hat{p},\Theta(\bar{u}_t)\rangle=\mu\|\Pi\left(\Theta(\bar{u}_t)\right)\|
\end{eqnarray}
Moreover, since $\Theta(u^*)\in-\mathbf{C}$ and $\bar{p}_t\in\mathbf{C}^*\cap\mathfrak{B}_{\mu}$, we have $\langle\bar{p}_t,\Theta(u^*)\rangle\leq0$. By~\eqref{star}, we have
\begin{eqnarray}\label{eq:rate1}
&&L(\bar{u}_t,\hat{p})-L(u^*,\bar{p}_t)\nonumber\\
&=&F(\bar{u}_t)+\langle\hat{p},\Theta(\bar{u}_t)\rangle-F(u^*)-\langle\bar{p}_t,\Theta(u^*)\rangle\nonumber\\
&\geq&F(\bar{u}_t)-F(u^*)+\langle\hat{p},\Theta(\bar{u}_t)\rangle\nonumber\\
&=&F(\bar{u}_t)-F(u^*)+\mu\|\Pi\left(\Theta(\bar{u}_t)\right)\|
\end{eqnarray}
Moreover, by taking $u=\bar{u}_t$ in the right hand side of saddle point inequality~\eqref{saddle point:L}, we have
\begin{eqnarray}\label{eq:rate_primal_8}
&&F(\bar{u}_t)-F(u^*)\nonumber\\
&\geq&-\langle p^*,\Theta(\bar{u}_t)\rangle\nonumber\\
&=&-\langle p^*,\Pi\left(\Theta(\bar{u}_t)\right)+\Pi_{-\mathbf{C}}\left(\Theta(\bar{u}_t)\right)\rangle\nonumber\\
&\geq&-\|p^*\|\|\Pi\left(\Theta(\bar{u}_t)\right)\|
\end{eqnarray}
Combine~\eqref{eq:rate1} and~\eqref{eq:rate_primal_8}, we have that
$$\|\Pi\left(\Theta(\bar{u}_t)\right)\|\leq\frac{L(\bar{u}_t,\hat{p})-L(u^*,\bar{p}_t)}{\left(\mu-\|p^*\|\right)}.$$
Take expectation on both side of above inequality, we have that
\begin{eqnarray}
\mathbb{E}_{\mathcal{F}_{t}}\|\Pi\left(\Theta(\bar{u}_t)\right)\|&\leq&\frac{\mathbb{E}_{\mathcal{F}_{t}}[L(\bar{u}_t,\hat{p})-L(u^*,\bar{p}_t)]}{\left(\mu-\|p^*\|\right)}\nonumber\\
&\leq&\mathbb{E}_{\mathcal{F}_{t}}\frac{Nh_3(u^*,\hat{p})}{\left(\mu-\|p^*\|\right)\underline{\epsilon}(t+1)}.\\
&&\qquad\qquad\mbox{(by statement (ii))}\nonumber
\end{eqnarray}
Since random variable $\hat{p}\in\mathfrak{B}_{\mu}$, it follows that
$$\mathbb{E}_{\mathcal{F}_{t}}\|\Pi\left(\Theta(\bar{u}_t)\right)\|\leq\frac{Nd_3}{\left(\mu-\|p^*\|\right)\underline{\epsilon}(t+1)},$$
where $d_3=\sup\limits_{\|p\|<\mu}h_3(u^*,p)$. The statement (iii) is provided.
\item[{\rm(iv)}] Again from~\eqref{eq:rate1},~\eqref{eq:rate_primal_8} and statement (iii), statement (iv) is coming.
\end{itemize}
\end{proof}
Observe that Theorem~\ref{thm:primal_rate} prompts SPDCL has the convergence rate $O(1/t)$. To obtain the dual suboptimality, we need the following additional assumption.
\begin{assumption}\label{assump3}
$G+J$ is coercive on $\mathbf{U}$ if $\mathbf{U}$ is not bounded, that is, $\forall \{u^{k}|k\in \mathbb{N}\}\subset \mathbf{U}$,
\begin{eqnarray*}
\lim_{k\rightarrow+\infty}\|u^{k}\|=+\infty\Rightarrow\lim_{k\rightarrow+\infty}(G+J)(u^{k})=+\infty.
\end{eqnarray*}
\end{assumption}
\indent The following lemma states that for any given bounded set of dual points, the corresponding optimizer of the augmented Lagrangian is bounded.
\begin{lemma}\label{lemma:ALBounded}
Suppose Assumption~\ref{assump1} holds. Let $\mathfrak{B}_p$ be a bounded set: $\mathfrak{B}_p=\{p\in\RR^m|\|p\|\leq d_p\}$. Then we have a positive constant $d_u$, for any $p\in\mathfrak{B}_p$, there is an optimizer $\hat{u}(p)\in\arg\min\limits_{u\in\mathbf{U}}L_{\gamma}(u,p)$ such that $\|\hat{u}(p)\|\leq d_u$.
\end{lemma}
\begin{proof} See~\cite{ZhaoZhu2017}.
\end{proof}
By statement (i) of Theorem~\ref{thm:primal_rate}, we have one ball: $\mathfrak{B}_{\nu}=\{u|\|u\|\leq\nu\}$ such that $\mathbb{E}_{\mathcal{F}_{t}}(\bar{u}_t)$ is contained in $\mathfrak{B}_{\nu}$. Furthermore, from Lemma~\ref{lemma:ALBounded} for $p\in\mathfrak{B}_{\mu}$ we have that there exists $\nu'>0$ such that $\hat{u}(p)=\arg\min L_\gamma(u,p)$ and $\|\hat{u}(p)\|\leq\nu'$. Specifically, we construct a new ball as $\mathfrak{B}_{{\nu}^{+}}=\{u|\|u\|\leq\overline{\nu}=\max(\nu,\nu')\}$. Next proposition shows that the pair of expected vectors $\big{(}\mathbb{E}_{\mathcal{F}_{t}}(\bar{u}_{t}), \mathbb{E}_{\mathcal{F}_{t}}(\bar{p}_{t})\big{)}$ is an approximate saddle point. This assertion will be used to derive the estimation on dual suboptimality for the average point $\bar{p}_t$.
\begin{proposition}{\bf (Approximate saddle points by expected point $\big{(}\mathbb{E}_{\mathcal{F}_{t}}(\bar{u}_t),\mathbb{E}_{\mathcal{F}_{t}}(\bar{p}_t)\big{)}$)}\label{prop:saddle point}\\
Suppose Assumptions of Theorem~\ref{thm:primal_rate} hold
\begin{itemize}
\item[(i)] Expected point $\big{(}\mathbb{E}_{\mathcal{F}_{t}}(\bar{u}_t),\mathbb{E}_{\mathcal{F}_{t}}(\bar{p}_t)\big{)}$ is an approximate saddle point for $L$: $\forall (u,p)\in(\mathbf{U}\cap\mathfrak{B}_{{\nu}^+})\times(\mathbf{C}^*\cap\mathfrak{B}_{\mu})$
\begin{eqnarray*}
-\frac{Nd_4}{\underline{\epsilon}(t+1)}+L\big{(}\mathbb{E}_{\mathcal{F}_{t}}(\bar{u}_{t},p)\big{)}
&\leq& L\big{(}\mathbb{E}_{\mathcal{F}_{t}}(\bar{u}_{t}),\mathbb{E}_{\mathcal{F}_{t}}(\bar{p}_{t})\big{)}\\
&\leq& L\big{(}u,\mathbb{E}_{\mathcal{F}_{t}}(\bar{p}_{t})\big{)}+\frac{Nd_4}{\underline{\epsilon}(t+1)}.
\end{eqnarray*}
where $d_4=\sup_{(u,p)\in(\mathbf{U}\cap\mathfrak{B}_{{\nu}^+})\times(\mathbf{C}^*\cap\mathfrak{B}_{\mu})}h_3(u,p)$.
\item[(ii)] Expected vectors $\big{(}\mathbb{E}_{\mathcal{F}_{t}}(\bar{u}_t),\mathbb{E}_{\mathcal{F}_{t}}(\bar{p}_t)\big{)}$ is an approximate saddle point for $L_\gamma$: $\forall (u,p)\in(\mathbf{U}\cap\mathfrak{B}_{{\nu}^+})\times(\mathbf{C}^*\cap\mathfrak{B}_{\mu})$
\begin{eqnarray*}
-\delta_1+L_{\gamma}\big{(}\mathbb{E}_{\mathcal{F}_{t}}(\bar{u}_{t}),p\big{)}&\leq& L_{\gamma}\big{(}\mathbb{E}_{\mathcal{F}_{t}}(\bar{u}_{t}),\mathbb{E}_{\mathcal{F}_{t}}(\bar{p}_{t})\big{)}\nonumber\\
&\leq& L_{\gamma}\big{(}u,\mathbb{E}_{\mathcal{F}_{t}}(\bar{p}_{t})\big{)}+\delta_2,\nonumber
\end{eqnarray*}
where $\delta_1=\frac{\mu Nd_3+\left(\mu-\|p^*\|\right)Nd_4}{\left(\mu-\|p^*\|\right)\underline{\epsilon}(t+1)}+\frac{\gamma N^2(d_3)^2}{2\left(\mu-\|p^*\|\right)^2\underline{\epsilon}^2(t+1)^2}$ and $\delta_2=\delta_1+\frac{Nd_4}{\underline{\epsilon}(t+1)}$.
\end{itemize}
\end{proposition}
\begin{proof}
\begin{itemize}
\item[(i)]From statement (ii) of Theorem~\ref{thm:primal_rate} with $u\in\mathbf{U}\cap\mathfrak{B}_{{\nu}^+}$ and $p\in\mathbf{C}^*\cap\mathfrak{B}_{\mu}$, we have that
\begin{eqnarray*}
\mathbb{E}_{\mathcal{F}_{t}}[L(\bar{u}_{t},p)-L(u,\bar{p}_{t})]\leq\frac{Nd_4}{\underline{\epsilon}(t+1)},
\end{eqnarray*}
where $d_4=\sup_{(u,p)\in(\mathbf{U}\cap\mathfrak{B}_{{\nu}^+})\times(\mathbf{C}^*\cap\mathfrak{B}_{\mu})}h_3(u,p)$. Since the bifunction $L(u',p)-L(u,p')$ is convex in $u'$ and linear in $p'$ for given $u\in\mathbf{U}$, $p\in\mathbf{C}^*$, we obtain
\begin{eqnarray}\label{saddle:L}
L\big{(}\mathbb{E}_{\mathcal{F}_{t}}(\bar{u}_{t}),p\big{)}-L\big{(}u,\mathbb{E}_{\mathcal{F}_{t}}(\bar{p}_{t})\big{)}\leq\frac{Nd_4}{\underline{\epsilon}(t+1)}.
\end{eqnarray}
Noted $\mathbb{E}_{\mathcal{F}_{t}}(\bar{p}_t)\in\mathbf{C}^*\cap\mathfrak{B}_{\nu}$, now with $p=\mathbb{E}_{\mathcal{F}_{t}}(\bar{p}_{t})$,~\eqref{saddle:L} yields the right inequality of approximate saddle point
\begin{eqnarray*}
L\big{(}\mathbb{E}_{\mathcal{F}_{t}}(\bar{u}_{t}),\mathbb{E}_{\mathcal{F}_{t}}(\bar{p}_t)\big{)}-L\big{(}u,\mathbb{E}_{\mathcal{F}_{t}}(\bar{p}_{t})\big{)}\leq\frac{Nd_4}{\underline{\epsilon}(t+1)},\\ \forall u\in\mathbf{U}\cap\mathfrak{B}_{{\nu}^+},
\end{eqnarray*}
Now considering $\mathbb{E}_{\mathcal{F}_{t}}(\bar{u}_{t})\in\mathbf{U}\cap\mathfrak{B}_{{\nu}^+}$, with $u=\mathbb{E}_{\mathcal{F}_{t}}(\bar{u}_{t})$,~\eqref{saddle:L} yields the left inequality
\begin{eqnarray*}
L\big{(}\mathbb{E}_{\mathcal{F}_{t}}(\bar{u}_{t}),p\big{)}-L\big{(}\mathbb{E}_{\mathcal{F}_{t}}(\bar{u}_{t}),\mathbb{E}_{\mathcal{F}_{t}}(\bar{p}_{t})\big{)}]\leq\frac{Nd_4}{\underline{\epsilon}(t+1)},\\ \forall p\in\mathbf{C}^*\cap\mathfrak{B}_{\mu}.
\end{eqnarray*}
Here comes the results.
\item[(ii)] In the left-hand side of inequality in statement (i), taking $p=0$, we get $\langle\mathbb{E}_{\mathcal{F}_{t}}(\bar{p}_t),\Theta(\mathbb{E}_{\mathcal{F}_{t}}(\bar{u}_t))\rangle\geq-\frac{Nd_4}{\underline{\epsilon}(t+1)}$. Then, from~\eqref{L_gamma}, we have
\begin{eqnarray}\label{saddle:Lr1}
\varphi\big{(}\Theta(\mathbb{E}_{\mathcal{F}_{t}}(\bar{u}_t)),\mathbb{E}_{\mathcal{F}_{t}}(\bar{p}_t)\big{)}&\geq&\langle\mathbb{E}_{\mathcal{F}_{t}}(\bar{p}_t),\Theta(\mathbb{E}_{\mathcal{F}_{t}}(\bar{u}_t))\rangle\nonumber\\
&\geq&-\frac{Nd_4}{\underline{\epsilon}(t+1)}.
\end{eqnarray}
Another hand, for $p\in\mathbf{C}^*\cap\mathfrak{B}_{\mu}$, we have
\begin{eqnarray}\label{saddle:Lr2}
&&\varphi\big{(}\Theta(\mathbb{E}_{\mathcal{F}_{t}}(\bar{u}_t)),p\big{)}\nonumber\\
&=&\min_{\xi\in-\mathbf{C}}\langle p,\Theta(\mathbb{E}_{\mathcal{F}_{t}}(\bar{u}_t))-\xi\rangle+\frac{\gamma}{2}\|\Theta(\mathbb{E}_{\mathcal{F}_{t}}(\bar{u}_t))-\xi\|^2\nonumber\\
&\leq&\langle p,\Theta(\mathbb{E}_{\mathcal{F}_{t}}(\bar{u}_t))-\Pi_{-\mathbf{C}}(\Theta(\mathbb{E}_{\mathcal{F}_{t}}(\bar{u}_t)))\rangle\nonumber\\
&&+\frac{\gamma}{2}\|\Theta(\mathbb{E}_{\mathcal{F}_{t}}(\bar{u}_t))-\Pi_{-\mathbf{C}}(\Theta(\mathbb{E}_{\mathcal{F}_{t}}(\bar{u}_t)))\|^2\nonumber\\
&\leq&\|p\|\cdot\|\Pi(\Theta(\mathbb{E}_{\mathcal{F}_{t}}(\bar{u}_t)))\|+\frac{\gamma}{2}\|\Pi(\Theta(\mathbb{E}_{\mathcal{F}_{t}}(\bar{u}_t)))\|^2\nonumber\\
&\leq&\frac{\mu Nd_3}{\left(\mu-\|p^*\|\right)\underline{\epsilon}(t+1)}+\frac{\gamma N^2(d_3)^2}{2\left(\mu-\|p^*\|\right)^2\underline{\epsilon}^2(t+1)^2}.\\
&&\mbox{(from $\|\Pi(\Theta(\mathbb{E}_{\mathcal{F}_{t}}(\bar{u}_t)))\|\leq\mathbb{E}_{\mathcal{F}_{t}}\|\Pi(\Theta(\bar{u}_t))\|$}\qquad\qquad\nonumber\\
&&\mbox{statement (ii) of Theorem~\ref{thm:primal_rate} and $p\in\mathbf{C}^*\cap\mathfrak{B}_{\mu}$)}\nonumber
\end{eqnarray}
Therefore, we get the left-hand side of inequality in statement (ii):
\begin{eqnarray}\label{eq:Lr_3}
&&L_{\gamma}(\mathbb{E}_{\mathcal{F}_{t}}(\bar{u}_t),p)-L_{\gamma}(\mathbb{E}_{\mathcal{F}_{t}}(\bar{u}_t),\mathbb{E}_{\mathcal{F}_{t}}(\bar{p}_t))\nonumber\\
&=&\varphi(\Theta(\mathbb{E}_{\mathcal{F}_{t}}(\bar{u}_t)),p)-\varphi(\Theta(\mathbb{E}_{\mathcal{F}_{t}}(\bar{u}_t),\mathbb{E}_{\mathcal{F}_{t}}(\bar{p}_t))\nonumber\\
&\leq&\delta_1,
\end{eqnarray}
where $\delta_1=\frac{\mu Nd_3+\left(\mu-\|p^*\|\right)Nd_4}{\left(\mu-\|p^*\|\right)\underline{\epsilon}(t+1)}+\frac{\gamma N^2(d_3)^2}{2\left(\mu-\|p^*\|\right)^2\underline{\epsilon}^2(t+1)^2}$. From~\eqref{saddle:Lr1} and~\eqref{saddle:Lr2}, it also has that
\begin{eqnarray*}
-\frac{Nd_4}{\underline{\epsilon}(t+1)}&\leq&\langle\mathbb{E}_{\mathcal{F}_{t}}(\bar{p}_t),\Theta(\mathbb{E}_{\mathcal{F}_{t}}(\bar{u}_t))\rangle\nonumber\\
&\leq&\varphi(\Theta(\mathbb{E}_{\mathcal{F}_{t}}(\bar{u}_t)),\mathbb{E}_{\mathcal{F}_{t}}(\bar{p}_t))\nonumber\\
&\leq&\frac{\mu Nd_3}{\left(\mu-\|p^*\|\right)\underline{\epsilon}(t+1)}\nonumber\\
&&+\frac{\gamma N^2(d_3)^2}{2\left(\mu-\|p^*\|\right)^2\underline{\epsilon}^2(t+1)^2},
\end{eqnarray*}
which follows that
\begin{eqnarray*} \varphi(\Theta(\mathbb{E}_{\mathcal{F}_{t}}(\bar{u}_t)),\mathbb{E}_{\mathcal{F}_{t}}(\bar{p}_t))-\langle\mathbb{E}_{\mathcal{F}_{t}}(\bar{p}_t),\Theta(\mathbb{E}_{\mathcal{F}_{t}}(\bar{u}_t))\rangle\leq\delta_1.
\end{eqnarray*}
Then, for $u\in\mathbf{U}\cap\mathfrak{B}_{{\nu}^+}$, we have
\begin{eqnarray}\label{eq:Lr_4}
&&L_\gamma(\mathbb{E}_{\mathcal{F}_{t}}(\bar{u}_t),\mathbb{E}_{\mathcal{F}_{t}}(\bar{p}_t))\nonumber\\
&\leq&L(\mathbb{E}_{\mathcal{F}_{t}}(\bar{u}_t),\mathbb{E}_{\mathcal{F}_{t}}(\bar{p}_t))+\delta_1\nonumber\\
&\leq&L(u,\mathbb{E}_{\mathcal{F}_{t}}(\bar{p}_t))+\delta_1+\frac{Nd_4}{\underline{\epsilon}(t+1)}\nonumber\\
&&\qquad\qquad\qquad\mbox{(by right hand side of statement (i))}\nonumber\\
&\leq&L_\gamma(u,\mathbb{E}_{\mathcal{F}_{t}}(\bar{p}_t))+\delta_2,
\end{eqnarray}
where $\delta_2=\delta_1+\frac{Nd_4}{\underline{\epsilon}(t+1)}$. Here comes the right-hand side of inequality in statement (ii).
\end{itemize}
\end{proof}
\begin{theorem}\label{thm:dual_rate} {\bf(Dual suboptimality)}
\noindent Let Assumptions of Theorem~\ref{thm:primal_rate} hold, we have that
\begin{equation}
\psi_{\gamma}(p^*)\leq\psi_{\gamma}\big{(}\mathbb{E}_{\mathcal{F}_{t}}(\bar{p}_t)\big{)}+\delta_1+\delta_2.
\end{equation}
\end{theorem}
\begin{proof}
For saddle point $(u^*,p^*)$ of $L$ (or $L_\gamma$) on $\mathbf{U}\times\RR^m$, we have
\begin{eqnarray}\label{eq:saddlepoint_Lr}
L_{\gamma}(u^*,p)\leq L_{\gamma}(u^*,p^*)\leq L_{\gamma}(u,p^*), \forall u\in\mathbf{U}, p\in\RR^m
\end{eqnarray}
Substituting $u=\mathbb{E}_{\mathcal{F}_{t}}(\bar{u}_t)$, $p=\mathbb{E}_{\mathcal{F}_{t}}(\bar{p}_t)$ in~\eqref{eq:saddlepoint_Lr}, and take $u=\hat{u}\big{(}\mathbb{E}_{\mathcal{F}_{t}}(\bar{p}_t)\big{)}$, $p=p^*$ in statement (ii) of Proposition~\ref{prop:saddle point}, we obtain the following two inequalities:
\begin{eqnarray*}
L_{\gamma}\big{(}u^*,\mathbb{E}_{\mathcal{F}_{t}}(\bar{p}_t)\big{)}&\leq&L_{\gamma}(u^*,p^*)\nonumber\\
&\leq&L_{\gamma}\big{(}\mathbb{E}_{\mathcal{F}_{t}}(\bar{u}_t),p^*\big{)}\label{eq:dual_1}\\
-\delta_1+L_{\gamma}\big{(}\mathbb{E}_{\mathcal{F}_{t}}(\bar{u}_{t}),p^*\big{)}&\leq&L_{\gamma}\big{(}\mathbb{E}_{\mathcal{F}_{t}}(\bar{u}_{t}),\mathbb{E}_{\mathcal{F}_{t}}(\bar{p}_{t})\big{)}\nonumber\\
&\leq&L_{\gamma}(\hat{u}(\mathbb{E}_{\mathcal{F}_{t}}\bar{p}_t),\mathbb{E}_{\mathcal{F}_{t}}\bar{p}_{t})+\delta_2.\label{eq:dual_2}
\end{eqnarray*}
Combining the above two inequalities, it follows the desired inequality:
\begin{eqnarray}
-\delta_1+L_{\gamma}(u^*,p^*)\leq L_{\gamma}\big{(}\hat{u}\big{(}\mathbb{E}_{\mathcal{F}_{t}}(\bar{p}_t)\big{)},\mathbb{E}_{\mathcal{F}_{t}}(\bar{p}_{t})\big{)}+\delta_2,
\end{eqnarray}
or
\begin{eqnarray}
\psi_\gamma(p^*)&=&L_{\gamma}(u^*,p^*)\nonumber\\
&\leq&L_{\gamma}\big{(}\hat{u}\big{(}\mathbb{E}_{\mathcal{F}_{t}}(\bar{p}_t)\big{)},\mathbb{E}_{\mathcal{F}_{t}}(\bar{p}_{t})\big{)}+\delta_1+\delta_2\nonumber\\
&=&\psi_\gamma\big{(}\mathbb{E}_{\mathcal{F}_{t}}(\bar{p}_t)\big{)}+\delta_1+\delta_2.
\end{eqnarray}
\end{proof}
Next we will provide the high probability complexity bound of constraints violation and objective function values.
\begin{remark}
From Theorem~\ref{thm:primal_rate}, we immediately get the expect primal suboptimality  for average point $\bar{u}_t$
\begin{eqnarray*}
&&\mathbb{E}_{\mathcal{F}_{t}}\bigg{\{}|F(\bar{u}_{t})-F(u^*)|+\|\Pi(\Theta(\bar{u}_{t}))\|\bigg{\}}\\
&\leq&\frac{(\mu+1)Nd_3}{(\mu-\|p^*\|)\underline{\epsilon}(t+1)}.
\end{eqnarray*}
Let $0<\varepsilon<|F(u^0)-F(u^*)|+\|\Pi(\Theta(u^0))\|$ and $\eta\in(0,1)$ be chosen arbitrarily. For all $t\geq T$, we have high probability complexity bound for obtaining an $\varepsilon$-optimal solution
\begin{equation*}
\mathbb{P}\big{\{}|F(\bar{u}_{t})-F(u^*)|+\|\Pi(\Theta(\bar{u}_{t}))\|\leq\varepsilon\big{\}}\geq 1-\eta,
\end{equation*}
where
\begin{eqnarray}
T:=\frac{(\mu+1)Nd_3}{\varepsilon\eta\underline{\epsilon}(\mu-\|p^*\|)}-1.
\end{eqnarray}
This result is derived from the Markov inequality~\cite{Bert2}. Another representation for this result is:\\
for any $t\geq T$
\begin{eqnarray*}
&&\mathbb{P}\big{\{}|F(\bar{u}_{t})-F(u^*)|+\|\Pi(\Theta(\bar{u}_{t}))\|\geq\varepsilon\big{\}}\nonumber\\
&\leq&\varepsilon^{-1}\mathbb{E}_{\mathcal{F}_{t}}\big{\{}|F(\bar{u}_{t})-F(u^*)|+\|\Pi(\Theta(\bar{u}_{t}))\|\big{\}}\nonumber\\
&\leq&\frac{(\mu+1)Nd_3}{\varepsilon(\mu-\|p^*\|)\underline{\epsilon}(t+1)}\nonumber\\
&\leq&\eta.
\end{eqnarray*}
\end{remark}
\begin{remark}
Here we remark that, for problem (P) with $\Omega(\cdot)=0$ and $\Theta(\cdot)=\Phi(\cdot)$, we modify SPDCL scheme as following:\\
\noindent\rule[0.25\baselineskip]{0.5\textwidth}{1.5pt}
{\bf Stochastic Primal-Dual Coordinate Method with Large step size (SPDCL)}\\
\noindent\rule[0.25\baselineskip]{0.5\textwidth}{0.5pt}
{Initialize} $u^0 \in \mathbf{U}$, $p^0\in \mathbf{R}^m$, and $0<\epsilon<\frac{\beta}{B_G+\gamma\tau^2}$  \\
\textbf{for} $k = 0,1,\cdots $, \textbf{do}
\begin{eqnarray}
&&\mbox{Choose $i(k)$ from $\{1,2,\ldots,N\}$ with equal probability}\nonumber\\
&&u^{k+1}\leftarrow\min_{u\in \mathbf{U}}\langle\nabla_{i(k)} G(u^{k}), u_{i(k)} \rangle + J_{i(k)}(u_{i(k)})\nonumber\\
&&+\langle\Pi(p^k+\gamma\Theta(u^k)),\Phi_{i(k)}(u_{i(k)})\rangle+\frac{1}{\epsilon}D(u,u^k);\label{primal_1}\\
&&p^{k+1}\leftarrow p^{k}+\frac{\rho}{\gamma}\big{(}\Pi(p^k+\gamma\Theta(u^{k+1}))-p^k\big{)},\qquad\label{dual_1}
\end{eqnarray}
\textbf{end for}\\
\noindent\rule[0.25\baselineskip]{0.5\textwidth}{1.5pt}
\indent Obviously, we don't need to estimate the dual optimal bound in the new scheme. Additionally,
using the constant parameter $0<\epsilon<\frac{\beta}{B_G+\gamma\tau^2}$, the results of Lemma~\ref{lemma:bound1} still holds. Therefore the results of convergence (Theorem~\ref{theo:convergence}) and convergence rate results (Theorem~\ref{thm:primal_rate} and~\ref{thm:dual_rate}) of SPDCL still hold.
\end{remark}
\section*{Appendix}
\par\noindent{\em Proof of Lemma~\ref{lemma:bound1}:}\\
{\rm(i)} Firstly, for all $u\in\mathbf{U}$, the unique solution $u^{k+1}$ of the primal problem~\eqref{eq:APk} is characterized by the following variational inequality:
\begin{eqnarray}\label{eq:primal_VI}
\langle\nabla_{i(k)}G(u^{k}),(u^{k+1}-u)_{i(k)}\rangle+J_{i(k)}(u_{i(k)}^{k+1})-J_{i(k)}(u_{i(k)})\nonumber\\
+\langle q^k,\nabla_{i(k)}\Omega(u^k)(u^{k+1}-u)_{i(k)}\rangle\nonumber\\
+\langle q^k,\Phi_{i(k)}(u_{i(k)}^{k+1})-\Phi_{i(k)}(u_{i(k)})\rangle\nonumber\\
+\frac{1}{\epsilon^k}\langle\nabla K(u^{k+1})-\nabla K(u^k), u^{k+1}-u\rangle\leq 0,
\end{eqnarray}
which follows that
\begin{eqnarray}\label{eq:primal_VI2}
\langle\nabla_{i(k)}G(u^{k}),\big{(}u^{k}-u-(u^{k}-u^{k+1})\big{)}_{i(k)}\rangle+J_{i(k)}(u_{i(k)}^{k})\nonumber\\
-J_{i(k)}(u_{i(k)})-\big{(}J_{i(k)}(u_{i(k)}^{k})-J_{i(k)}(u_{i(k)}^{k+1})\big{)}\nonumber\\
+\langle q^k,\nabla_{i(k)}\Omega(u^k)\big{(}u^{k}-u-(u^{k}-u^{k+1})\big{)}_{i(k)}\rangle\nonumber\\
+\langle q^k,\Phi_{i(k)}(u_{i(k)}^{k})-\Phi_{i(k)}(u_{i(k)})\nonumber\\
-\big{(}\Phi_{i(k)}(u_{i(k)}^{k})-\Phi_{i(k)}(u_{i(k)}^{k+1})\big{)}\rangle\nonumber\\
+\frac{1}{\epsilon^k}\langle\nabla K(u^{k+1})-\nabla K(u^k), u^{k+1}-u\rangle\leq 0.
\end{eqnarray}
Observing that\\
$\langle\nabla_{i(k)}G(u^{k}),(u^{k}-u^{k+1})_{i(k)}\rangle=\langle\nabla G(u^{k}),u^{k}-u^{k+1}\rangle$,\\ $J_{i(k)}(u_{i(k)}^{k})-J_{i(k)}(u_{i(k)}^{k+1})=J(u^{k})-J(u^{k+1})$,\\
$\langle q^k,\nabla_{i(k)}\Omega(u^k)(u^{k}-u^{k+1})_{i(k)}\rangle=\langle q^k,\nabla\Omega(u^k)(u^{k}-u^{k+1})\rangle$\\
and $\langle q^k,\Phi_{i(k)}(u_{i(k)}^{k})-\Phi_{i(k)}(u_{i(k)}^{k+1})\rangle=\langle q^k,\Phi(u^{k})-\Phi(u^{k+1})\rangle$, from~\eqref{eq:primal_VI2}, we have that
\begin{eqnarray}\label{eq:primal_bound0}
&&\langle\nabla_{i(k)}G(u^{k}),(u^{k}-u)_{i(k)}\rangle+J_{i(k)}(u_{i(k)}^{k})\nonumber\\
&&-J_{i(k)}(u_{i(k)})+\langle q^k,\nabla_{i(k)}\Omega(u^k)(u^{k}-u)_{i(k)}\rangle\nonumber\\
&&+\langle q^k,\Phi_{i(k)}(u_{i(k)}^{k})-\Phi_{i(k)}(u_{i(k)})\rangle\nonumber\\
&\leq&\langle\nabla G(u^{k}),u^{k}-u^{k+1}\rangle+J(u^{k})-J(u^{k+1})\nonumber\\
&&+\langle q^k,\nabla\Omega(u^k)(u^{k}-u^{k+1})+\Phi(u^{k})-\Phi(u^{k+1})\rangle\nonumber\\
&&+\frac{1}{\epsilon^k}\langle\nabla K(u^{k+1})-\nabla K(u^k), u-u^{k+1}\rangle
\end{eqnarray}
By statement (ii) and (iii) of Lemma~\ref{lemma:Lipschitz3point}, we have that
\begin{eqnarray}\label{eq:primal_bound1}
&&\langle\nabla_{i(k)}G(u^{k}),(u^{k}-u)_{i(k)}\rangle+J_{i(k)}(u_{i(k)}^{k})\nonumber\\
&&-J_{i(k)}(u_{i(k)})+\langle q^k,\nabla_{i(k)}\Omega(u^k)(u^{k}-u)_{i(k)}\rangle\nonumber\\
&&+\langle q^k,\Phi_{i(k)}(u_{i(k)}^{k})-\Phi_{i(k)}(u_{i(k)})\rangle\nonumber\\
&\leq&F(u^k)-F(u^{k+1})+\frac{B_G+\|q^k\|T}{2}\|u^{k}-u^{k+1}\|^2\nonumber\\
&&+\langle q^k,\Theta(u^{k})-\Theta(u^{k+1})\rangle\nonumber\\
&&+\frac{1}{\epsilon^k}\langle\nabla K(u^{k+1})-\nabla K(u^k), u-u^{k+1}\rangle.
\end{eqnarray}
The simple algebraic operation and Assumption~\ref{assump2} follows that
\begin{eqnarray}\label{eq:primal_bound2}
&&\frac{1}{\epsilon^k}\langle \nabla K(u^{k+1})-\nabla K(u^k),u-u^{k+1}\rangle\nonumber\\
&=&\frac{1}{\epsilon^k}\big{[}D(u,u^k)-D(u,u^{k+1})-D(u^{k+1},u^k)\big{]}\nonumber\\
&\leq&\frac{1}{\epsilon^k}\big{[}D(u,u^k)-D(u,u^{k+1})\big{]}\nonumber\\
&&-\frac{\beta}{2\epsilon^k}\|u^k-u^{k+1}\|^2.
\end{eqnarray}
Combining~\eqref{eq:primal_bound1} and~\eqref{eq:primal_bound2}, we obtain that
\begin{eqnarray}\label{eq:primal_bound3}
&&\langle\nabla_{i(k)}G(u^{k}),(u^{k}-u)_{i(k)}\rangle+J_{i(k)}(u_{i(k)}^{k})-J_{i(k)}(u_{i(k)})\nonumber\\
&&+\langle q^k,\nabla_{i(k)}\Omega(u^k)(u^{k}-u)_{i(k)}\nonumber\\
&&+\Phi_{i(k)}(u_{i(k)}^{k})-\Phi_{i(k)}(u_{i(k)})\rangle\nonumber\\
&\leq&\frac{1}{\epsilon^k}\big{[}D(u,u^k)-D(u,u^{k+1})\big{]}\nonumber\\
&&+F(u^k)-F(u^{k+1})+\langle q^k,\Theta(u^{k})-\Theta(u^{k+1})\rangle\nonumber\\
&&-\frac{\beta-\epsilon^k(B_G+\|q^k\|T)}{2\epsilon^k}\|u^k-u^{k+1}\|^2.
\end{eqnarray}
Take expectation with respect to $i(k)$ on both side of~\eqref{eq:primal_bound3}, together the condition expectation~\eqref{expectationG}-\eqref{expectationP_2}, we get
\begin{eqnarray}\label{eq:primal_bound3_1}
&&\frac{1}{N}\big{[}F(u^{k})-F(u)+\langle q^k,\Theta(u^{k})-\Theta(u)\rangle\big{]}\nonumber\\
&\leq&\frac{1}{\epsilon^k}\big{[}D(u,u^k)-\mathbb{E}_{i(k)}D(u,u^{k+1})\big{]}\nonumber\\
&&+\mathbb{E}_{i(k)}\bigg{\{}\big{[}F(u^k)-F(u^{k+1})+\langle q^k,\Theta(u^{k})-\Theta(u^{k+1})\rangle\big{]}\nonumber\\
&&-\frac{\beta-\epsilon^k(B_G+\|q^k\|T)}{2\epsilon^k}\|u^k-u^{k+1}\|^2\bigg{\}}.
\end{eqnarray}
It follows that
\begin{eqnarray}\label{eq:primal_bound4}
&&\frac{1}{N}\mathbb{E}_{i(k)}\big{[}F(u^{k+1})-F(u)+\langle q^k,\Theta(u^{k+1})-\Theta(u)\rangle\big{]}\nonumber\\
&\leq&\frac{1}{\epsilon^k}\big{[}D(u,u^k)-\mathbb{E}_{i(k)}D(u,u^{k+1})\big{]}\nonumber\\
&&+\mathbb{E}_{i(k)}\bigg{\{}\frac{N-1}{N}\big{[}F(u^k)-F(u^{k+1})\nonumber\\
&&+\langle q^k,\Theta(u^{k})-\Theta(u^{k+1})\rangle\big{]}\nonumber\\
&&-\frac{\beta-\epsilon^k(B_G+\|q^k\|T)}{2\epsilon^k}\|u^k-u^{k+1}\|^2\bigg{\}}
\end{eqnarray}
By $\nabla_{\theta}\varphi(\Theta(u),p)=\Pi(p+\gamma\Theta(u))$ in Theorem~\ref{theo:Lag_2}. Then it follows that
\begin{eqnarray}
&&\|\nabla_{\theta}\varphi(\Theta(u),p)-\nabla_{\theta}\varphi(\Theta(\hat{u}),p)\|\nonumber\\
&=&\|\Pi(p+\gamma\Theta(u))-\Pi(p+\gamma\Theta(\hat{u}))\|\nonumber\\
&\leq&\gamma\|\Theta(u)-\Theta(\hat{u})\|.
\end{eqnarray}
Together with statement (ii) of Lemma~\ref{lemma:Lipschitz3point}, we have that
\begin{eqnarray}\label{eq:primal_bound4-1}
&&\langle q^k,\Theta(u^{k})-\Theta(u^{k+1})\rangle\nonumber\\
&\leq&\varphi(\Theta(u^k),p^k)-\varphi(\Theta(u^{k+1}),p^k)\nonumber\\
&&+\frac{\gamma}{2}\|\Theta(u^k)-\Theta(u^{k+1})\|^2\nonumber\\
&\leq&\varphi(\Theta(u^k),p^k)-\varphi(\Theta(u^{k+1}),p^k)+\frac{\gamma\tau^2}{2}\|u^k-u^{k+1}\|^2\nonumber\\
\end{eqnarray}
Combining~\eqref{eq:primal_bound4} and~\eqref{eq:primal_bound4-1}, we have that
\begin{eqnarray}\label{eq:primal_bound5}
&&\frac{1}{N}\mathbb{E}_{i(k)}\big{[}F(u^{k+1})-F(u)+\langle q^k,\Theta(u^{k+1})-\Theta(u)\rangle\big{]}\nonumber\\
&\leq&\frac{1}{\epsilon^k}\big{[}D(u,u^k)-\mathbb{E}_{i(k)}D(u,u^{k+1})\big{]}\nonumber\\
&&+\frac{N-1}{N}\mathbb{E}_{i(k)}\big{[}L_{\gamma}(u^k,p^k)-L_{\gamma}(u^{k+1},p^k)\nonumber\\
&&+\frac{\gamma\tau^2}{2}\|u^k-u^{k+1}\|^2\big{]}\nonumber\\
&&-\frac{\beta-\epsilon^k(B_G+\|q^k\|T)}{2\epsilon^k}\mathbb{E}_{i(k)}\|u^k-u^{k+1}\|^2\nonumber\\
&=&\frac{1}{\epsilon^k}\big{[}D(u,u^k)-\mathbb{E}_{i(k)}D(u,u^{k+1})\big{]}\nonumber\\
&&+\frac{N-1}{N}\mathbb{E}_{i(k)}\left[L_{\gamma}(u^k,p^k)-L_{\gamma}(u^{k+1},p^{k+1})\right]\nonumber\\
&&+\frac{N-1}{N}\mathbb{E}_{i(k)}\left[\varphi\big{(}\Theta(u^{k+1}),p^{k+1}\big{)}-\varphi\big{(}\Theta(u^{k+1}),p^{k}\big{)}\right]\nonumber\\
&&-\frac{\beta-\epsilon^k(B_G+\|q^k\|T+\frac{N-1}{N}\gamma\tau^2)}{2\epsilon^k}\mathbb{E}_{i(k)}\|u^k-u^{k+1}\|^2\nonumber\\
\end{eqnarray}
From concavity of $\varphi\big{(}\Theta(u),p\big{)}$ in $p$ and statement (ii) of Theorem~\ref{theo:Lag_2}, the third term of~\eqref{eq:primal_bound5} follows that
\begin{eqnarray}\label{eq:primal_bound5-1}
&&\varphi\big{(}\Theta(u^{k+1}),p^{k+1}\big{)}-\varphi\big{(}\Theta(u^{k+1}),p^{k}\big{)}\nonumber\\
&\leq&\frac{1}{\gamma}\langle q^{k+1/2}-p^k, p^{k+1}-p^k\rangle\nonumber\\
&\leq&\frac{1}{\gamma}\|q^{k+1/2}-p^k\|\cdot\|p^{k+1}-p^k\|\nonumber\\
&\leq&\frac{1}{\gamma}\|q^{k+1/2}-p^k\|\cdot\|p^{k}+\frac{\rho}{\gamma}(q^{k+1/2}-p^k)-p^k\|\nonumber\\
&&\qquad\qquad\mbox{(since dual update~\eqref{dual} and $p^k\in\mathfrak{B}_{\mu}$)}\nonumber\\
&=&\frac{\rho}{\gamma^2}\|q^{k+1/2}-p^k\|^2.
\end{eqnarray}
Together~\eqref{eq:primal_bound5} and inequality~\eqref{eq:primal_bound5-1}, we have that
\begin{eqnarray}\label{eq:primal_bound5-3}
&&\frac{1}{N}\mathbb{E}_{i(k)}\big{[}F(u^{k+1})-F(u)+\langle q^k,\Theta(u^{k+1})-\Theta(u)\rangle\big{]}\nonumber\\
&\leq&\frac{1}{\epsilon^k}\big{[}D(u,u^k)-\mathbb{E}_{i(k)}D(u,u^{k+1})\big{]}\nonumber\\
&&+\frac{N-1}{N}\mathbb{E}_{i(k)}\left[L_{\gamma}(u^k,p^k)-L_{\gamma}(u^{k+1},p^{k+1})\right]\nonumber\\
&&+\frac{(N-1)\rho}{N\gamma^2}\mathbb{E}_{i(k)}\|q^{k+1/2}-p^k\|^2\nonumber\\
&&-\frac{\beta-\epsilon^k(B_G+\|q^k\|T+\frac{N-1}{N}\gamma\tau^2)}{2\epsilon^k}\mathbb{E}_{i(k)}\|u^k-u^{k+1}\|^2\nonumber\\
\end{eqnarray}
Multiply $\epsilon^k$ on both side of~\eqref{eq:primal_bound5-3}, by the definition of $\Lambda(u,p,u^k,p^k)$, statement (i) is provided.\\
\\
{\rm(ii)} In order to prove statement (ii), we first derive two inequalities. By the property~\eqref{eq:Projecproperty3} of projection with $y=p\in\mathbf{C}^*$ and $x=p^k+\gamma\Theta(u^{k+1})$ ,we have
\begin{eqnarray}\label{eq:dual_bound1}
\frac{1}{\gamma}\langle p-q^{k+1/2},p^k+\gamma\Theta(u^{k+1})-q^{k+1/2}\rangle\leq0.
\end{eqnarray}
Using Proposition~\ref{proposition} with $x=\gamma\Theta(u^{k+1})$, $y=\gamma\Theta(u^k)$ and $z=p^k$, we have
\begin{eqnarray}\label{eq:dual_bound2}
&&2\langle q^{k+1/2}-q^k,\gamma\Theta(u^{k+1})\rangle\nonumber\\
&\leq&\gamma^2\tau^2\|u^k-u^{k+1}\|^2+\|q^{k+1/2}-p^k\|^2-\|q^k-p^k\|^2.\nonumber\\
\end{eqnarray}
For all $p\in\mathbf{C}^*$, from~\eqref{eq:dual_bound1}, it follows:
\begin{eqnarray}\label{eq:dual_bound2-1}
&&L(u^{k+1},p)-L(u^{k+1},q^k)\nonumber\\
&=&\langle p-q^k,\Theta(u^{k+1})\rangle\nonumber\\
&=&\frac{1}{\gamma}\bigg{[}\langle p-q^{k+1/2},p^k+\gamma\Theta(u^{k+1})-q^{k+1/2}\rangle\nonumber\\
&&+\langle p-q^{k+1/2}, q^{k+1/2}-p^k\rangle\nonumber\\
&&+\langle q^{k+1/2}-q^k,\gamma\Theta(u^{k+1})\rangle\bigg{]}\nonumber\\
&\leq&\frac{1}{\gamma}\bigg{[}\langle p-q^{k+1/2}, q^{k+1/2}-p^k\rangle\nonumber\\
&&+\langle q^{k+1/2}-q^k,\gamma\Theta(u^{k+1})\rangle\bigg{]}.
\end{eqnarray}
Together~\eqref{eq:dual_bound2-1} and~\eqref{eq:dual_bound2}, we have
\begin{eqnarray}\label{eq:dual_bound2-1-1}
&&L(u^{k+1},p)-L(u^{k+1},q^k)\nonumber\\
&=&\frac{1}{\gamma}\bigg{[}\langle p-p^k, q^{k+1/2}-p^k\rangle-\|q^{k+1/2}-p^k\|^2\nonumber\\
&&+\langle q^{k+1/2}-q^k,\gamma\Theta(u^{k+1})\rangle\bigg{]}\nonumber\\
&\leq&\frac{1}{\gamma}\langle p-p^k, q^{k+1/2}-p^k\rangle-\frac{1}{2\gamma}\|q^{k+1/2}-p^k\|^2\nonumber\\
&&-\frac{1}{2\gamma}\|q^k-p^k\|^2+\frac{\gamma\tau^2}{2}\|u^k-u^{k+1}\|^2
\end{eqnarray}
Since $p^k,p^{k+1}\in\mathfrak{B}_{\mu}$, we have that: $\forall p\in\mathbf{C}^*\cap\mathfrak{B}_{\mu}$,
\begin{eqnarray}\label{eq:dual_bound2-2}
&&\|p^k-p\|^2-\|p^{k+1}-p\|^2\nonumber\\
&\geq&\|p^k-p\|^2-\|p^{k}+\frac{\rho}{\gamma}(q^{k+1/2}-p^k)-p\|^2\nonumber\\
&=&2\langle p-p^k,\frac{\rho}{\gamma}(q^{k+1/2}-p^k)\rangle-\frac{\rho^2}{\gamma^2}\|q^{k+1/2}-p^k\|^2\nonumber\\
\end{eqnarray}
Together~\eqref{eq:dual_bound2-1} and~\eqref{eq:dual_bound2-2}, we have that: $\forall p\in\mathbf{C}^*\cap\mathfrak{B}_{\mu}$,
\begin{eqnarray}\label{eq:dual_bound3}
&&L(u^{k+1},p)-L(u^{k+1},q^k)\nonumber\\
&\leq&\frac{1}{2\rho}\left[\|p^k-p\|^2-\|p^{k+1}-p\|^2\right]\nonumber\\
&&+\frac{\rho-\gamma}{2\gamma^{2}}\|q^{k+1/2}-p^k\|^2\nonumber\\
&&-\frac{1}{2\gamma}\|q^k-p^k\|^2+\frac{\gamma\tau^2}{2}\|u^k-u^{k+1}\|^2
\end{eqnarray}
Multiply $\frac{\epsilon^k}{N}$ on both side of above inequality, by $\rho=\frac{\gamma}{2N-1}$ we obtain that: $\forall p\in\mathbf{C}^*\cap\mathfrak{B}_{\mu}$
\begin{eqnarray}\label{eq:dual_bound4}
&&\frac{\epsilon^k}{N}\big{[}L(u^{k+1},p)-L(u^{k+1},q^k)\big{]}\nonumber\\
&=&\frac{\epsilon^k}{2N\rho}\big{[}\|p-p^k\|^2-\|p-p^{k+1}\|^2\big{]}\nonumber\\
&&+\frac{(1-N)\rho\epsilon^k}{N\gamma^{2}}\|q^{k+1/2}-p^k\|^2\nonumber\\
&&-\frac{\epsilon^k}{2N\gamma}\|q^k-p^k\|^2+\frac{\epsilon^k\frac{1}{N}\gamma\tau^2}{2}\|u^k-u^{k+1}\|^2.
\end{eqnarray}
Statement (ii) is provided by take expectation with respect to $i(k)$ on both side of inequality~\eqref{eq:dual_bound4}.\\
\\
{\rm(iii)} Summing the two inequalities in statement (i) and statement (ii), we have that
\begin{eqnarray*}
&&\frac{\epsilon^k}{N}\mathbb{E}_{i(k)}\big{[}L(u^{k+1},p)-L(u,q^{k})\big{]}\\
&\leq&\Lambda^k(u,p,u^k,p^k)-\mathbb{E}_{i(k)}\Lambda^k(u,p,u^{k+1},p^{k+1})\\
&&-\mathbb{E}_{i(k)}\big{[}\frac{\beta-\epsilon^k(B_G+\|q^k\|T+\gamma\tau^2)}{2}\|u^k-u^{k+1}\|^2\\
&&-\frac{\epsilon^k}{2N\gamma}\|q^k-p^k\|^2\big{]}.
\end{eqnarray*}
Since the SPDCL scheme guarantees that
$$\epsilon^k\leq\frac{\beta}{2(B_G+\|q^k\|T+\gamma\tau^2)},$$
then we have the statement (iii).\\
\\
{\rm(iv)} From~\eqref{eq:primal_bound3_1}, we have that
\begin{eqnarray}\label{eq:sum_bound_1}
&&\frac{1}{N}\big{[}L(u^{k},q^k)-L(u,q^k)\big{]}\nonumber\\
&=&\frac{1}{N}\big{[}F(u^{k})-F(u)+\langle q^k,\Theta(u^{k})-\Theta(u)\rangle\big{]}\nonumber\\
&\leq&\frac{1}{\epsilon^k}\big{[}D(u,u^k)-\mathbb{E}_{i(k)}D(u,u^{k+1})\big{]}\nonumber\\
&&+\mathbb{E}_{i(k)}\big{[}F(u^k)-F(u^{k+1})+\langle q^k,\Theta(u^{k})-\Theta(u^{k+1})\rangle\big{]}\nonumber\\
&&-\frac{\beta-\epsilon^k(B_G+\|q^k\|T)}{2\epsilon^k}\|u^k-u^{k+1}\|^2\qquad\mbox{(by~\eqref{eq:primal_bound3_1})}\nonumber\\
&=&\frac{1}{\epsilon^k}\mathbb{E}_{i(k)}[K(u)-K (u^{k})-\langle\nabla K(u^k),u-u^k\rangle\nonumber\\
&&-K(u)+K(u^{k+1})+\langle\nabla K(u^{k+1}),u-u^{k+1}]\nonumber\\
&&+\mathbb{E}_{i(k)}\big{[}F(u^k)-F(u^{k+1})+\langle q^k,\Theta(u^{k})-\Theta(u^{k+1})\rangle\big{]}\nonumber\\
&&-\frac{\beta-\epsilon^k(B_G+\|q^k\|T)}{2\epsilon^k}\|u^k-u^{k+1}\|^2\nonumber\\
&=&\frac{1}{\epsilon^k}\mathbb{E}_{i(k)}[K(u^{k+1})-K (u^{k})-\langle\nabla K(u^{k}),u^{k+1}-u^{k}\rangle\nonumber\\
&&+\langle\nabla K(u^{k+1})-\nabla K(u^k),u-u^{k+1}\rangle]\nonumber\\
&&+\mathbb{E}_{i(k)}\big{[}F(u^k)-F(u^{k+1})+\langle q^k,\Theta(u^{k})-\Theta(u^{k+1})\rangle\big{]}\nonumber\\
&&-\frac{\beta-\epsilon^k(B_G+\|q^k\|T)}{2\epsilon^k}\|u^k-u^{k+1}\|^2\nonumber\\
&\leq&\mathbb{E}_{i(k)}\bigg{[}\frac{B}{\epsilon^k}\big{(}\frac{1}{2}\|u^k-u^{k+1}\|^2+\|u-u^{k+1}\|\|u^k-u^{k+1}\|\big{)}\nonumber\\
&&+\big{(}\|\nabla G(u^k)\|+c_1\|u^k\|+c_2+\tau \|q^k\|\big{)}\|u^k-u^{k+1}\|\nonumber\\
&&-\frac{\beta-\epsilon^k(B_G+\|q^k\|T)}{2\epsilon^k}\|u^k-u^{k+1}\|^2\bigg{]}.
\end{eqnarray}
From~\eqref{eq:dual_bound3}, we have that
\begin{eqnarray}\label{eq:sum_bound_2}
&&\frac{1}{N}\big{[}L(u^{k},p)-L(u^{k},q^k)\big{]}\nonumber\\
&=&\frac{1}{N}\langle p-q^k,\Theta(u^k)\rangle\nonumber\\
&=&\frac{1}{N}[\langle p-q^k,\Theta(u^{k+1})\rangle+\langle p-q^k,\Theta(u^k)-\Theta(u^{k+1})\rangle]\nonumber\\
&\leq&\frac{1}{2N\rho}\big{[}\|p-p^k\|^2-\|p-p^{k+1}\|^2\big{]}\nonumber\\
&&+\frac{(1-N)\rho}{N\gamma^{2}}\|q^{k+1/2}-p^k\|^2-\frac{1}{2N\gamma}\|q^k-p^k\|^2\nonumber\\
&&+\frac{\frac{1}{N}\gamma\tau^2}{2}\|u^k-u^{k+1}\|^2\nonumber\\
&&+\frac{1}{N}\langle p-q^k,\Theta(u^k)-\Theta(u^{k+1})\rangle\nonumber\\
&\leq&\frac{1}{2N\rho}\big{[}\|p-p^k\|^2-\|p-p^{k+1}\|^2\big{]}+\frac{\frac{1}{N}\gamma\tau^2}{2}\|u^k-u^{k+1}\|^2\nonumber\\
&&+\frac{1}{N}\langle p-q^k,\Theta(u^k)-\Theta(u^{k+1})\rangle
\end{eqnarray}
Since\\
$\|p-p^k\|^2-\|p-p^{k+1}\|^2=\langle2p-p^{k+1}-p^k,p^{k+1}-p^k\rangle\leq\|2p-p^{k+1}-p^k\|\cdot\|p^{k+1}-p^k\|$\\
and $\langle p-q^k,\Theta(u^k)-\Theta(u^{k+1})\rangle\leq\tau\|p-q^k\|\cdot\|u^k-u^{k+1}\|$, we have
\begin{eqnarray}\label{eq:sum_bound_2-1}
&&\frac{1}{N}\big{[}L(u^{k},p)-L(u^{k},q^k)\big{]}\nonumber\\
&\leq&\frac{1}{2N\rho}\|2p-p^{k+1}-p^k\|\cdot\|p^{k+1}-p^k\|\nonumber\\
&&+\frac{\frac{1}{N}\gamma\tau^2}{2}\|u^k-u^{k+1}\|^2\nonumber\\
&&+\frac{\tau}{N}\|p-q^k\|\cdot\|u^k-u^{k+1}\|.
\end{eqnarray}
Take expectation with respect to $i(k)$ on both side of~\eqref{eq:sum_bound_2-1} and sum with~\eqref{eq:sum_bound_1}, we obtain that
\begin{eqnarray}\label{eq:sum_bound_3}
&&\frac{1}{N}\big{[}L(u^{k},p)-L(u,q^k)\big{]}\nonumber\\
&\leq&\mathbb{E}_{i(k)}\bigg{[}\frac{B}{2\epsilon^k}\|u^k-u^{k+1}\|^2\nonumber\\
&&+h_1(\epsilon^k,u,p,u^k,u^{k+1},q^k)\|u^k-u^{k+1}\|\nonumber\\
&&+h_2(p,p^k,p^{k+1})\|p^{k+1}-p^k\|\nonumber\\
&&-\frac{\beta-\epsilon^k(B_G+\|q^k\|T+\frac{1}{N}\gamma\tau^2)}{2\epsilon^k}\|u^k-u^{k+1}\|^2\bigg{]}\nonumber\\
&\leq&\mathbb{E}_{i(k)}\big{[}\frac{B}{2\epsilon^k}\|u^k-u^{k+1}\|^2\nonumber\\
&&+h_1(\epsilon^k,u,p,u^k,u^{k+1},q^k)\|u^k-u^{k+1}\|\nonumber\\
&&+h_2(p,p^k,p^{k+1})\|p^{k+1}-p^k\|\big{]}.
\end{eqnarray}
where $h_1(\epsilon^k,u,p,u^k,u^{k+1},q^k)=\frac{B}{\epsilon^k}\|u-u^{k+1}\|+[\|\nabla G(u^k)\|+c_1\|u^k\|+c_2+\tau \|q^k\|]+\frac{\tau}{N}\|p-q^k\|$
 and $h_2(p,p^k,p^{k+1})=\frac{1}{2N\rho}\|2p-p^{k+1}-p^k\|$.
\endproof
\section*{Acknowledgment}

The authors would like to thank...

\ifCLASSOPTIONcaptionsoff
  \newpage
\fi

\begin{IEEEbiographynophoto}{Daoli Zhu}
Biography text here.
\end{IEEEbiographynophoto}
\begin{IEEEbiographynophoto}{Lei Zhao}
Biography text here.
\end{IEEEbiographynophoto}




\end{document}